\documentclass[reqno,a4paper,tbtags]{amsart}

\usepackage[margin=25mm]{geometry}

\usepackage{amsaddr}
\usepackage{amssymb}
\usepackage{mathtools}
\usepackage[shortlabels]{enumitem}
\usepackage{dsfont}
\usepackage{xcolor}
\usepackage{booktabs,multirow}
\usepackage{bm}
\usepackage[labelfont=bf,labelformat=simple]{subcaption}
	
\usepackage[noadjust]{cite}
\usepackage{hyperref}

\newtheorem{theorem}{Theorem}[section]

\newtheorem{proposition}[theorem]{Proposition}
\newtheorem{corollary}[theorem]{Corollary}
\theoremstyle{definition}

\theoremstyle{remark}
\newtheorem{remark}[theorem]{Remark}

\numberwithin{equation}{section}

\newcommand*{\set}[1]{\left\{#1\right\}}
\newcommand{\real}{\mathbb R}
\newcommand{\ex}{\mathsf E}

\DeclareMathOperator{\Cov}{Cov}

\newcommand{\bs}[1]{}

\allowdisplaybreaks

\begin{document}

\title[Fractional Gaussian noise: projections, prediction, norms]{Fractional Gaussian noise: projections, prediction, norms}

\author{Iryna Bodnarchuk$^1$, Yuliya Mishura$^{1}$ and Kostiantyn Ralchenko$^{1,2}$}
\address{$^1$Taras Shevchenko National University of Kyiv, Ukraine\\
$^2$University of Vaasa, Finland}

\email{ibodnarchuk@knu.ua, yuliyamishura@knu.ua, kostiantynralchenko@knu.ua}

\thanks{YM is supported by The Swedish Foundation for Strategic Research, grant UKR22-0017, and by the Japan Science and Technology Agency CREST, project reference number JPMJCR2115.
KR is supported by the Research Council of Finland, decision number 359815.
YM and KR acknowledge that the present research is
carried out within the frame and support of the ToppForsk project no.
274410 of the Research Council of Norway with the title STORM: Stochastics
for Time-Space Risk Models.}

\begin{abstract}
We examine the one-sided and two-sided (bilateral) projections of an element of fractional Gaussian noise onto its neighboring elements. We establish several analytical results and conduct a numerical study to analyze the behavior of the coefficients of these projections as functions of the Hurst index and the number of neighboring elements used for the projection. We derive recurrence relations for the coefficients of the two-sided projection. Additionally, we explore the norms of both types of projections. Certain special cases are investigated in greater detail, both theoretically and numerically.
\end{abstract}



\maketitle

%%%%
\section{Introduction}

Consider a fractional Brownian motion (fBm) $B^H=\{B_t^H, t\ge 0\}$ with Hurst index $H\in(0,1)$. That is, $B^H$ is a centered Gaussian process with covariance function of the form
\begin{gather}\label{fBmCov}
R(t,s) \coloneqq \ex B_t^HB_s^H= \frac{1}{2}\left(t^{2H}+s^{2H}-|t-s|^{2H}\right).
\end{gather}

Let
$$
\Delta_k=B_k^H-B_{k-1}^H,\quad k\ge 1,
$$
be the $k$th increment of fBm taken in subsequent integer points $k\ge 1$. Then we get from  \eqref{fBmCov} that
\begin{gather}\label{rhok}
 \rho_{k}=\ex \Delta_1\Delta_{k+1}=\frac12\left(|k+1|^{2H}-2|k|^{2H}+|k-1|^{2H}\right),\quad k\ge 1,\quad \rho_{0}=1.
\end{gather} Due to the stationarity of the increments,
\begin{gather*}
 \ex \Delta_{k}\Delta_{l}=\rho_{|k-l|},\quad k,l\ge 1.
\end{gather*}
These subsequent increments $
 \Delta_k$, $k\ge 1$, create a process that is named fractional Gaussian noise (fGn). The main properties of this relatively simple discrete-time Gaussian process are stationarity and the presence of what we call memory. The length of memory is infinite, however, its intensity depends on Hurst index $H$, which in turn should be the object of statistical estimation. For the properties of fGn and statistical estimation   see, e.g.\ \cite{beran, deli, liu, mandelbrot1, mandelbrot2, meerson, qian, wang} (of course, any  list of references is not exhaustive). These properties have made fractional Gaussian noise extremely popular in applications, in particular,  to physics (\cite{li1, li2}), hydrology (\cite{molz}), information theory (\cite{stratonovich}),   signal detection (\cite{barton}), related permutation entropy (\cite{davalos, zunino}) and many other fields. However, considerable analytical and computational difficulties arise in those problems in which the covariance matrices of fBm and fGn  and their determinants are involved. The reason for this is obvious, it lies in the huge number of various fractional powers presented in covariance function and covariance matrices, and the source of this
is  the fractional Hurst index, see \eqref{rhok}. For more description of related problems, see \cite{malyarenko,  MishuraRalchenkoSchilling2022, risks}.

%%%%%%%%%%%%%%%%%%%%%%%%%%%%%%%%%%%%%%%%%%%%%%
In particular, the paper \cite{risks} contains three open problems related to the covariance matrices of fBm and fGn.
To formulate these problems, let us define, for $n\ge1$, the triangular array $\set{d_{j,k}, 1\le j\le k\le n}$ by the following relation:
\begin{equation}\label{eq:chol}
\sum_{j=1}^{p \wedge r}d_{j,p}d_{j,r} = R(p,r),
\quad p,r\in\set{1,\dots,n}.
\end{equation}
The sequence $\set{d_{j,k}, 1\le j\le k\le n}$ exists and is unique, as  \eqref{eq:chol} represents the Cholesky decomposition of the covariance matrix of fBm.
%Let $G_n=(\rho_{|j-k|})_{j,k=1}^n$ denote the covariance matrix of $(\Delta_1,\dots,\Delta_n)$.
Similarly to \eqref{eq:chol}, we can define the Cholesky decomposition of the covariance matrix of fGn as follows
\[
\sum_{j=1}^{p \wedge r}\ell_{j,p}\ell_{j,r} = \rho_{|p-r|},
\quad p,r\in\set{1,\dots,n}.
\]
The properties of the sequences $\set{d_{j,k}}$ and $\set{\ell_{j,k}}$ for $H\in(\frac12,1)$ were investigated in \cite{risks}.
In particular, the positivity of both sequences was established, along with the monotonicity of $\set{d_{j,k}}$ with respect to $k$ for a fixed $j$.
In their study of this problem, the authors of \cite{risks} found a connection between the projection  coefficients, that is, the coefficients of the one-sided projection of any value of a stationary Gaussian process onto finitely many subsequent elements, and the Cholesky decomposition of the covariance matrix of the process. More precisely, according to the theorem of normal correlation,  there exists a real-valued sequence
$\{\Gamma_{n}^{k}, 2\le k\le n\}$
such that
\begin{gather}\label{projectgamma1}
 \ex (\Delta_1\mid\Delta_2,\ldots, \Delta_n)
 =\sum_{k=2}^{n}\Gamma_{n}^k\Delta_k.
\end{gather}
For any $n\ge1$, the coefficients $\{\Gamma_{n}^{k}, 2\le k\le n\}$ can be computed as a solution to the following linear system of equations
\begin{equation}\label{eq:gamma-sys}
\rho_{l-1} = \sum_{k=2}^n \Gamma_n^k \rho_{|l-k|},
    \quad 2\le l\le n.
\end{equation}
The properties of the coefficients $\Gamma_{n}^k$ were further investigated in \cite{MishuraRalchenkoSchilling2022}, where recurrence relations for them were obtained, see \eqref{eq:Gamma-last}--\eqref{eq:Gamma-k} below.

Moreover, the following open problems were posed in \cite{risks} as conjectures.

\begin{enumerate}[itemindent=81pt,leftmargin=*,label=\textbf{Conjecture A\arabic*.},ref=A\arabic*]
\item\label{hyp:cols}
For all $r\ge1$, $d_{1,r}> d_{2,r} > \dots > d_{r,r}$.

\item\label{hyp:diag}
For all $1\le j\le k$, $\ell_{j,k} > \ell_{j+1,k+1}$.

\item\label{hyp:sol}
The coefficients $\Gamma_n^k$ for $n = 2, 3, \dots$, $k =
2, 3, \dots, n$, are strictly positive.
\end{enumerate}

It was shown in \cite{risks} that Conjecture \ref{hyp:sol} implies Conjecture \ref{hyp:diag}, which in turn implies  Conjecture \ref{hyp:cols}. Also, Conjecture \ref{hyp:sol} was confirmed in \cite{MishuraRalchenkoSchilling2022, risks} numerically for a wide range of values of $n$. Note also, that due to stationarity of fGn,   coefficients of one-sided projection can be considered as the prediction coefficients, because
\[
\ex (\Delta_n\mid\Delta_1,\ldots, \Delta_{n-1})
 =\sum_{k=1}^{n-1}\Gamma_{n}^{n-k+1}\Delta_k.
\]

To understand better the properties of projection coefficients for other Gaussian--Volterra noises, we considered in \cite{bodnmish} a very simple process of the form
\begin{equation*}%\label{Xt}
    X_t=\int_0^t(t-s) \,dW_s,
\end{equation*}
where $W$ is a Wiener process.

In \cite{bodnmish} we establish that $X$, like fBm, is self-similar, non-Markov, has a long memory, its increments over non-overlapping intervals are positively correlated. But, unlike fBm, its increments are not stationary.
The projection problem of the form~\eqref{projectgamma1}  for the process $X$ was considered in~\cite{bodnmish}.
Using a combinatorial approach, we obtained the explicit formulas for the respective projection coefficients. Note that this is apparently one of the few cases when the coefficients can be calculated explicitly.
We established that the coefficients are not all positive, moreover, they are alternating. Thus we can assume that stationarity or non-stationarity of the  increments is precisely the property that is determining the signs of projection coefficients. But for now, this statement is still a hypothesis.

With all these previous results in mind, in this paper we considered three main tasks: to proceed analytically with the properties of the coefficients of one-sided projection, to investigate the coefficients of the two-sided (bilateral) projection and to study the norms of both kinds of projections as the functions of $n$ and $H$. Along the way, we made a rather unexpected observation: while all the coefficients of one-sided projection remain positive, at least within the limits of our observations, with a two-sided projection one of the coefficients steadily becomes negative, but quite small in absolute value. We have established this property analytically for a small value of $n$, although this ``exceptional'' behavior  seems somewhat strange and inexplicable, from a logical point of view.

The paper is organized as follows: Section \ref{sec2} is devoted to some  properties of the coefficients of one-sided projections. More precisely, we establish that all coefficients of one-sided projection~\eqref{projectgamma1} in the case $n=4$ are strictly positive and at the same time we show what technical difficulties will arise along the way and why we really limit ourselves to small values of  $n$ in precise  calculations.
Then we consider another form of projection that contains orthogonal summands,  calculate the coefficients of such projection and give a simple equality for the $L_2$-norm of one-sided projection. In Sections \ref{sec3} and \ref{sec4} we calculate the coefficients of bilateral projection and comment in detail their (somehow unexpected) properties. Section \ref{sec5} is devoted to the norms of projections. We see that the norms are stabilized after some, not very big, value of $n$ and both of them do not tend to 1. It means that, from the point of view of the theory of stationary sequences, fractional Gaussian noise is purely nondeterministic sequence.

%%%%%%%%%%%%%%%%%%%%%%%%%%%%%%%%%%%%%%%%%%%%%%%%%%%%
\section{Some results concerning coefficients of one-sided projection}\label{sec2}
As it was explained in the Introduction, Conjecture~A3 about the sign of the coefficients of one-sided projection~\eqref{projectgamma1} (namely, the hypothesis that they are strictly positive) has not been proved analytically yet. However, we can a bit proceed in this direction, and demonstrate simultaneously what technical difficulties appear on this way in the general case, in comparison with \cite{MishuraRalchenkoSchilling2022}.

%%%%%%
\subsection{Coefficients of one-sided projection \texorpdfstring{\eqref{projectgamma1}}{(1.4)} in the case \texorpdfstring{$\bm{n=4}$}{n=4}}

Let $n=4$. Then
\begin{gather*}
 \ex (\Delta_{1}\mid\Delta_{2},\Delta_{3},\Delta_{4})
 =\Gamma_4^2\Delta_2+\Gamma_4^3\Delta_3+\Gamma_4^4\Delta_4.
\end{gather*}

It was proved in \cite{MishuraRalchenkoSchilling2022} that
$\Gamma_n^2>0$ for any $n\ge 2$, and it was established in \cite[Proposition~3]{MishuraRalchenkoSchilling2022} that for any $H\in(1/2,1)$, $\Gamma_4^2>\Gamma_4^3$ and $\Gamma_4^4>0$.

Positivity of $\Gamma_4^3$ was established in \cite{MishuraRalchenkoSchilling2022} numerically.
Now we establish it analytically. According to equality~(20) from \cite{MishuraRalchenkoSchilling2022}
\begin{gather}\label{gamma43}
 \Gamma_4^3=\frac{\rho_1^2\rho_2-\rho_2^3+\rho_1\rho_2\rho_3-\rho_1^2+\rho_2-\rho_1\rho_3}{1+2\rho_1^2\rho_2-\rho_2^2-2\rho_1^2},
\end{gather}
and the denominator in the right-hand side of \eqref{gamma43}, being a determinant of covariance matrix, is strictly positive.
Therefore, it is sufficient to prove that
\begin{gather}\label{gamma43>0}
 \rho_1^2\rho_2-\rho_2^3+\rho_1\rho_2\rho_3-\rho_1^2+\rho_2-\rho_1\rho_3>0.
\end{gather}

It was mentioned in \cite[Remark 5]{MishuraRalchenkoSchilling2022}
(and it is very easy to see by direct transformations) that the left-hand side of \eqref{gamma43>0} equals $(1-\rho_2)(\rho_2+\rho_2^2-\rho_1^2-\rho_1\rho_3)$. Also, all coefficients $\rho_k<1$. Therefore, it is sufficient to prove that
\begin{gather*}\label{rhohat}
 \hat{\rho}:=\rho_2+\rho_2^2-\rho_1^2-\rho_1\rho_3>0.
\end{gather*}
We have that $\rho_2>\rho_1^2$ and $\rho_2^2<\rho_1\rho_3$
(it was established in Lemma~3 and Corollary~1 of \cite{MishuraRalchenkoSchilling2022}, respectively). This means that the direct derivation for the  sign of \eqref{rhohat} is not obvious. Therefore, let us simply substitute the values of $\rho_1,\rho_2,\rho_3$ and proceed with the formulas containing corresponding powers. It is very easy to see that as a function of $H$, $\hat{\rho}$ has a form
$$
\hat{\rho}=\hat{\rho}(H)= \frac14\left( 9^{2H}-8^{2H}-2\cdot 6^{2H}+4\cdot 4^{2H}-2\cdot 2^{2H}-1 \right ),
$$
and $\hat{\rho}=0$ if $H=1/2$ and $H=1$.

\begin{proposition}
 $\hat{\rho}(H)>0$ for all $H\in(1/2,1).$
\end{proposition}
\begin{proof}
 Denote $2H=x\in(1,2)$, and redesignate
 $$
 \tilde{\rho}(x)= 4 \hat{\rho}(H)=9^{x}-8^{x}-2\cdot 6^{x}+4\cdot 4^{x}-2\cdot 2^{x}-1
 =(3^x-2^x)^2 -8^{x} +3\cdot 4^x-2\cdot 2^{x}-1.
 $$

Furthermore, denote $3^x=u$ and $2^x=v$, forgetting for the moment about their connection. Then $u\in [3;9]$ and $v\in [2;4]$.

Consider the function of two variables
$$
F(u,v)=(u-v)^2  -v^3+3v^2-2v-1,\quad (u,v)\in[3;9]\times[2;4].
$$

Then
$$
\frac{\partial F}{\partial u}=2(u-v),\quad \mbox{and} \quad \frac{\partial F}{\partial v}=-2(u-v)-3v^2+6v-2.
$$

Let's find the points where $\frac{\partial F}{\partial u}=\frac{\partial F}{\partial v}=0$. Then $u=v$ and consequently $u=v= 1\pm1/\sqrt3<2$. Therefore, there is no such points at the rectangle $[3;9]\times[2;4]$ and so, the smallest value of $F(u,v)$ is achieved on the boundary of this rectangle. But, remembering the relationship between $u$ and $v$,  on the boundary $u=3$ we have $v=2$, and at this point
$ \tilde{\rho}(1)=0$, and on the boundary $u=9$ we have that $v=4$ and again $\tilde{\rho}=\tilde{\rho}(2)=0$ while, for example, at point $x=3/2$
$$
\tilde{\rho}(x)=\tilde{\rho}(3/2)
=58-20\sqrt{2}-12\sqrt{6}\approx 0.3219>0.
$$

It means that minimal value of $\tilde{\rho}(x)$ is achieved at points $x=1$ and $x=2$ being  equal  zero, whence the proof follows.
\end{proof}

%%%%%%
\subsection{Some ``conditional'' relations}  Now our goal is to consider general value of $n$ and construct some kind of recurrence relations.
According to \cite[Proposition~6]{MishuraRalchenkoSchilling2022},
the first coefficient $\Gamma_4^2>0$ for all $n\ge 2$.
Now, assume that we already proved that $\Gamma_n^k>0$ for some $n\ge 2$ and any $2\le k\le n$.

\begin{proposition}
 If we know that $\Gamma_n^k>0$ for some $n\ge 2$ and any $2\le k\le n$, then the last coefficient
 $\Gamma_{n+1}^{n+1}$ in the expansion \eqref{projectgamma1} for $n+1$ is positive: $\Gamma_{n+1}^{n+1}>0$.
\end{proposition}
\begin{proof}
Coefficients $\{\Gamma_{n}^{k}\in\real,\ 2\le k\le n\}$ in the expansion \eqref{projectgamma1} are determined recursively in \cite[Proposition~5]{MishuraRalchenkoSchilling2022}. Namely,
\begin{gather}
 \Gamma_{n+1}^{n+1}=\frac{\rho_n-\sum_{k=2}^n\Gamma_n^k\rho_{n+1-k}}{1-\sum_{k=2}^n\Gamma_n^k\rho_{k-1}},\quad n\ge 2,
 \label{eq:Gamma-last}
 \\
 \Gamma_{n+1}^{k}=\Gamma_{n}^{k}-\Gamma_{n+1}^{n+1}\Gamma_{n}^{n-k+2}, \quad n\ge 2,\ 2\le k \le n.
 \label{eq:Gamma-k}
\end{gather}

The denominator in \eqref{eq:Gamma-last} is also strictly positive, as a determinant of covariance matrix.
Therefore, it is sufficient to prove that
$\rho_n-\sum_{k=2}^n\Gamma_n^k\rho_{n+1-k}>0$,
under assumption that $\Gamma_n^k>0$ for all $2\le k\le n$.

Multiplying both parts of \eqref{projectgamma1} by $\Delta_n$ and taking expectation, we get that
$$
\rho_{n-1}-\sum_{k=2}^n\Gamma_n^k\rho_{n-k}=0.
$$

Therefore, it is sufficient to prove that
$$
\delta_n:=\rho_{n-1}-\rho_n-\sum_{k=2}^n\Gamma_n^k(\rho_{n-k}-\rho_{n+1-k})<0.
$$

However,
$$
\delta_n=\rho_{n-1}\left(1-\frac{\rho_n}{\rho_{n-1}}\right)
-\sum_{k=2}^n\Gamma_n^k\rho_{n-k}\left(1-\frac{\rho_{n+1-k}}{\rho_{n-k}}\right),
$$
and taking to the account that $\Gamma_n^k>0$ and also $0<\rho_k<\rho_{k-1} \le 1$, $k\ge 1$ (see \cite[Corollary~1]{MishuraRalchenkoSchilling2022}), we see that it is sufficient to prove that
$$
1-\frac{\rho_n}{\rho_{n-1}}<1-\frac{\rho_{n+1-k}}{\rho_{n-k}}\quad \mbox{for }\quad 2\le k\le n.
$$

However, this relation is a direct consequence of inequality (13) from \cite{MishuraRalchenkoSchilling2022} which states that
the coefficients $\rho_k$ are log-convex, and so $$\frac{\rho_l}{\rho_{l-1}}<\frac{\rho_{l+1}}{\rho_l}.$$
%In turn, it is sufficient to prove that the function
%$$
%g(x)=\frac{\rho_{x}}{\rho_{x+1}}
%=\frac{(x+1)^{2H}-2x^{2H}+(x-1)^{2H}}{(x+2)^{2H}-2(x+1)^{2H}+x^{2H}}
%$$
%increases in $x\ge 2$. Taking logarithm, we see that it is sufficient to prove that
%$\log \rho_x-\log \rho_{x+1}$ increases in $x$, or, that is the same, $\log \rho_{x+1}-\log \rho_{x}$ decreases in $x$.
%The latter relation is equivalent to
%$\psi(x):=\frac{\rho_x'}{\rho_x}$ decreases in $x$. But
%$$
%\psi(x)=2H\frac{(x+1)^{2H-1}-2x^{2H-1}+(x-1)^{2H-1}}{(x+1)^{2H}-2x^{2H}+(x-1)^{2H}},
%$$
%and the denominator $\rho_x=(x+1)^{2H}-2x^{2H}+(x-1)^{2H}$ decreases in $x\ge 1$.
%Furthermore, since $2H-1\in (0,1)$, numerator $(x+1)^{2H-1}-2x^{2H-1}+(x-1)^{2H-1}$ is strictly negative.
%Therefore, it is sufficient to prove that
%$-\psi(x)$ increases in $x$, being strictly positive.
%It is sufficient to prove that
%$(x-1)^{2H-1}-x^{2H-1}-\left(x^{2H-1}-(x+1)^{2H-1}\right)$ increases in $x\ge 1$. It will be the case if
%$x^{2H-1}-(x-1)^{2H-1}$ decreases in $x$, again, it is equivalent to the fact that $(x^{2H-1})''<0$, but
%$$
%(x^{2H-1})''=(2H-1)(2H-2)x^{2H-3}<0.
%$$

Proposition is proved.
\end{proof}

\begin{remark}
 Since we already proved that for $n=4$ $\Gamma_n^k>0,\ 2\le k\le 4$, it means that $\Gamma_5^5>0$. However, to prove analytically that $\Gamma_5^4>0$ is a much more tedious problem than to prove that $\Gamma_4^3>0$ (both coefficients are ``penultimate''), therefore it is better to prove this fact numerically, see Figure~6 in \cite{MishuraRalchenkoSchilling2022}.
 The situation with the next coefficients is even more involved. This explains why to get the general result
 $\Gamma_n^k>0$, $n\ge2$, $2\le k\le n$, is indeed problematic.
\end{remark}

%%%%%%
\subsection{``Martingale'' approach to the calculation of coefficients}
Obviously, the following conditional expectations are equal:
\begin{multline*}
 \ex (\Delta_1\mid\Delta_2,\ldots, \Delta_n)\\*
 =\ex (\Delta_1\mid\Delta_2,
 \Delta_3- \ex (\Delta_3\mid\Delta_2),% \Delta_4- \ex (\Delta_4\mid\Delta_2,\Delta_3),
 \ldots,
 \Delta_k- \ex (\Delta_k\mid\Delta_2,\ldots,\Delta_{k-1}),
 \ldots, \Delta_n- \ex (\Delta_n\mid\Delta_2,\ldots,\Delta_{n-1})),
\end{multline*}
  all random variables in both conditions are Gaussian, and in the right-hand side they are non-correlated (orthogonal), therefore create a martingale according to the filtration $\mathcal{F}_k=\sigma\{\Delta_2,\ldots,\Delta_k\}$, and   they are even independent.
Consequently,
\begin{gather*}
 \ex (\Delta_1\mid\Delta_2,\ldots, \Delta_n)
 =\sum_{k=3}^{n}R_n^k(\Delta_k- \ex (\Delta_k\mid\Delta_2,\ldots,\Delta_{k-1}))+R_n^2\Delta_2,
\end{gather*}
for some coefficients $R_n^k\in\real, n\ge 2, 2\le k\le n$.

\begin{proposition}
 Coefficients $R_n^k$ do not depend on $n$ and equal
 \begin{gather}
  R_n^k=R_k=\frac{\rho_{k-1}-\sum_{i=2}^{k-1}\Gamma_{k-1}^i\rho_{k-i}}{1-\sum_{i=2}^{k-1}\Gamma_{k-1}^i\rho_{i}},\quad k\ge 3,\label{Rk}\\
  R_2=\rho_1.\label{R2}
 \end{gather}
\end{proposition}
\begin{proof}
 It follows from the pairwise   orthogonality of the terms $\Delta_k- \ex (\Delta_k\mid\Delta_2,\ldots,\Delta_{k-1})$ that
 $$
 R_n^k=\frac{\ex\left(\ex (\Delta_1\mid\Delta_2,\ldots, \Delta_n)(\Delta_k- \ex (\Delta_k\mid\Delta_2,\ldots,\Delta_{k-1}))\right)}
 {\ex(\Delta_k- \ex (\Delta_k\mid\Delta_2,\ldots,\Delta_{k-1}))^2}.
 $$

 Recall also that
 $\ex (\Delta_1\mid\Delta_2,\ldots, \Delta_{k-1})
 =\sum_{i=2}^{k-1}\Gamma_{k-1}^i\Delta_i$. Taking into account stationarity of fGn, we can  rewrite the latter equality ``symmetrically'':
 $$
 \ex (\Delta_k\mid\Delta_2,\ldots, \Delta_{k-1})
 =\sum_{i=2}^{k-1}\Gamma_{k-1}^i\Delta_{k-i+1},
 $$
 and consequently, for any $k\ge 3$,
 \begin{align}
  \ex\bigl(\ex (\Delta_1\mid\Delta_2,\ldots, \Delta_n)(\Delta_k- \ex (\Delta_k\mid\Delta_2,\ldots,\Delta_{k-1}))\bigr)
  &=\ex \left(\Delta_1\left(\Delta_k-\sum_{i=2}^{k-1}
  \Gamma_{k-1}^i\Delta_{k-i+1}
  \right)\right)\nonumber\\
  \label{Rknum}
  &=\rho_{k-1}-\sum_{i=2}^{k-1}
  \Gamma_{k-1}^i\rho_{k-i}.
 \end{align}

 Furthermore,
\begin{gather}\label{Rkdenom}
 \ex(\Delta_k- \ex (\Delta_k\mid\Delta_2,\ldots,\Delta_{k-1}))^2
 =1-\ex\Delta_k\ex (\Delta_k\mid\Delta_2,\ldots, \Delta_{k-1})
 =1-\sum_{i=2}^{k-1}
  \Gamma_{k-1}^i\rho_{i-1}.
\end{gather}
Equalities \eqref{Rknum} and \eqref{Rkdenom} imply \eqref{Rk}. Equality \eqref{R2} is obvious.
\end{proof}

\begin{corollary}
 Again, it immediately follows from the orthogonality of summands that the $L_2$-norm of one-sided projection equals
 $$
 R_1(n):=\ex|\ex (\Delta_1\mid\Delta_2,\ldots, \Delta_{n})|^2=\sum_{k=2}^{n}R_k^2.
 $$
 Despite the simplicity of the form, this equality contains implicitly the coefficients $\Gamma_n^k$, so again it is not so unambiguous for calculations.
\end{corollary}

%%%%%%%%%%%%%%%%%%%%%%%%%%%%%%%%%%%%%%%%%%%%%%%%%%%%

\section{Formulas for calculating coefficients of bilateral projection}\label{sec3}

In this section we proceed with two-sided (bilateral) projections of fractional Gaussian noise. Let us consider the following family of random variables, where each of them  symbolizes the  bilateral  projection of one element of fGn on the symmetric two-sided family of its elements:
\begin{gather*}
 H_n^{j}=\ex (\Delta_{n}\mid\Delta_{n-j},\ldots,\Delta_{n-1},\Delta_{n+1},\ldots,\Delta_{n+j}),\quad n\ge 2, \quad 1\le j\le n-1.
\end{gather*}

Due to the stationarity of fractional Gaussian noise  and the theorem of normal correlation (see, e.g. \cite[Theorem 13.1]{LiptserShiryaev2013} or \cite[Proposition 1.2]{MishuraZili2018}) we obtain the presentation
\begin{gather}\label{projectj}
H_n^{j}=\sum_{k=n-j, k\neq n}^{n+j}Q_{j}^{|n-k|}\Delta_k
=\sum_{k=1}^{j}Q_{j}^{k}\left(\Delta_{n-k}+\Delta_{n+k}\right),
\end{gather}
where $Q_{j}^{k}\in\real,\ 1\le k\le j\le n-1$ are the respective projection coefficients.

Note  that stationarity of fractional Gaussian noise also implies   that
$
Q_{j}^{k}=\hat{Q}_{j}^{k},
$
for all $1\le k\le j\le \min\set{n-1,m-1}$,
where
$$
 H_m^{j}=\sum_{k=m-j, k\neq m}^{m+j}\hat{Q}_{j}^{|m-k|}\Delta_k=\sum_{k=1}^{j}\hat{Q}_{j}^{k}\left(\Delta_{m-k}+\Delta_{m+k}\right).
$$

Let any $n\ge 2$ be fixed. Our aim is to find the coefficients $Q_{n-1}^{k}\in\real,\ 1\le k\le n-1$, of the decomposition
\begin{gather}\label{projectn}
H_n^{n-1}=\ex(\Delta_{n}\mid\Delta_{1},\ldots,\Delta_{n-1},\Delta_{n+1},\ldots,\Delta_{2n-1})
=\sum_{k=1, k\neq n}^{2n-1}Q_{n-1}^{|n-k|}\Delta_k.
\end{gather}

In order to realize this plan, we multiply both sides of \eqref{projectn} by
$\Delta_l$, for all $1\le l\le 2n-1, l\neq n$, and take the expectations. Obviously,
$$\ex(\ex(\Delta_{n}\mid\Delta_{1},\ldots,\Delta_{n-1},\Delta_{n+1},\ldots,\Delta_{2n-1})\Delta_l)=\rho_{|n-l|}.$$
As a result,  we obtain the following system of linear equations:
\begin{gather*}%\label{systqlong}
 \rho_{|n-l|}=
 \sum_{k=1, k\neq n}^{2n-1}Q_{n-1}^{|n-k|}\rho_{|k-l|},
 \quad 1\le l\le 2n-1,\ l\neq n,
\end{gather*}
or  the following  equivalent system of linear equations:
\begin{gather}\label{systq}
 \rho_{n-l}=
 \sum_{k=1}^{n-1}Q_{n-1}^{n-k}\left(\rho_{|k-l|}+\rho_{|2n-k-l|}\right),
 \quad 1\le l\le n-1.
\end{gather}

We can rewrite these systems in matrix form. Accordingly, they will take on the form
\begin{gather}\label{vecteq}
 \overline{\rho}_{2n-2}=A_{2n-2}\overline{Q}_{2n-2},\quad
 \overline{\rho}_{n-1}^*=A_{n-1}^*\overline{Q}_{n-1}^*,
\end{gather}
where
\begin{gather}
 \overline{\rho}_{2n-2}=
 \begin{pmatrix}
        \rho_{n-1}\\
        \rho_{n-2}\\
        \vdots\\
        \rho_{1}\\
        \rho_{1}\\
        \vdots\\
        \rho_{n-2}\\
        \rho_{n-1}
\end{pmatrix},
\
\overline{Q}_{2n-2}=
 \begin{pmatrix}
        Q_{n-1}^{n-1}\\
        Q_{n-1}^{n-2}\\
        \vdots\\
        Q_{n-1}^{1}\\
        Q_{n-1}^{1}\\
        \vdots\\
        Q_{n-1}^{n-2}\\
        Q_{n-1}^{n-1}
\end{pmatrix},
\
\overline{\rho}_{n-1}^*=
 \begin{pmatrix}
        \rho_{n-1}\\
        \rho_{n-2}\\
        \vdots\\
        \rho_{1}
\end{pmatrix},
\
\overline{Q}_{n-1}^*=
 \begin{pmatrix}
        Q_{n-1}^{n-1}\\
        Q_{n-1}^{n-2}\\
        \vdots\\
        Q_{n-1}^{1}
\end{pmatrix},
\nonumber\\
A_{2n-2}=
\begin{pmatrix}
1 & \rho_1 & \cdots  & \rho_{n-2} & \rho_{n} & \rho_{n+1} & \cdots  &\rho_{2n-3} & \rho_{2n-2}\\
\rho_{1} & 1 & \cdots  & \rho_{n-3} & \rho_{n-1} & \rho_{n} & \cdots  &\rho_{2n-4} & \rho_{2n-3}\\
\vdots  & \vdots  & \ddots  & \vdots & \vdots & \vdots  & \ddots  & \vdots & \vdots\\
\rho_{2n-3} & \rho_{2n-4} & \cdots  & \rho_{n-1} & \rho_{n-2} & \rho_{n-3} & \cdots  & 1 & \rho_{1}\\
\rho_{2n-2} & \rho_{2n-3} & \cdots  & \rho_{n} & \rho_{n-1} & \rho_{n-2} & \cdots  &\rho_{1} & 1
\end{pmatrix},\nonumber\\
\label{matrA*n-1}
A_{n-1}^*=
\begin{pmatrix}
1+\rho_{2n-2} & \rho_{1}+\rho_{2n-3} & \cdots  & \rho_{n-3}+\rho_{n+1} & \rho_{n-2}+\rho_{n}\\
\rho_{1}+\rho_{2n-3} & 1+\rho_{2n-4} & \cdots  & \rho_{n-4}+\rho_{n} & \rho_{n-3}+\rho_{n-1}\\
\vdots  & \vdots  & \ddots  & \vdots & \vdots \\
\rho_{n-3}+\rho_{n+1} & \rho_{n-4}+\rho_{n} & \cdots  & 1+\rho_{4} & \rho_{1}+\rho_{3}\\
\rho_{n-2}+\rho_{n} & \rho_{n-3}+\rho_{n-1} & \cdots  & \rho_{1}+\rho_{3} & 1+\rho_{2}
\end{pmatrix}.
\end{gather}

%%%%%%%%%%% Theorem: Coefficients
\begin{theorem}\label{theor2.1}
 Coefficients $\{Q_{n-1}^{k}\in\real,\ 1\le k\le n-1\}$ can be calculated  using  the following formulas:

 1) For $n=2$  the unique respective coefficient equals
\begin{gather}\label{coef11}
 Q_1^1=\frac{\rho_{1}}{1+\rho_2}.
\end{gather}

 2) For $n\ge 3$ the respective coefficients equal
\begin{gather}
\label{Qn-1n-1}
 Q_{n-1}^{n-1}
 =\frac{\rho_{n-1}-\sum_{k=2, k\neq n}^{2n-1}G_{2n-1}^{k}\rho_{|n-k|}}{1-\sum_{k=2, k\neq n}^{2n-1}G_{2n-1}^{k}\rho_{k-1}},\\
 \label{Qn-1k}
 Q_{n-1}^{k}=Q_{n-2}^{k}-Q_{n-1}^{n-1}
 \left(T_{n-2}^{n-k} +T_{n-2}^{n+k}\right),\quad 1\le k\le n-2,
\end{gather}
where
\begin{gather}\label{G2n-1}
 G_{2n-1}^{k}=\Gamma_{2n-1}^{k}+\Gamma_{2n-1}^{n}S_{2n-1}^{k},\quad 2\le k \le 2n-1,\ k\neq n,
 \\
 \label{S2n-1}
 S_{2n-1}^{2n-1}
 =\frac{\rho_{n-1}-\sum_{k=2, k\neq n}^{2n-2}T_{n-2}^{2n-k}\rho_{|n-k|}}{1-\sum_{k=2, k\neq n}^{2n-2}T_{n-2}^{2n-k}\rho_{|2n-1-k|}},
 \\
 \label{Sk}
 S_{2n-1}^{k}=Q_{n-2}^{|n-k|}-S_{2n-1}^{2n-1}T_{n-2}^{2n-k},
 \quad 2\le k \le 2n-2,\ k\neq n,
 \\
 \label{Tn-2k}
 T_{n-2}^{k}=\Gamma_{2n-2}^{k}+\Gamma_{2n-2}^{n}Q_{n-2}^{|n-k|},
 \quad 2\le k \le 2n-2,\ k\neq n.
\end{gather}
\end{theorem}
%%%%%%%%%%%%%

\begin{proof}
1) Let $n=2$. In this case,
the system \eqref{systq} is reduced to one equation, namely, to
\begin{gather*}
 \rho_{1}=Q_1^1(1+\rho_2),
\end{gather*}
and we immediately get \eqref{coef11}.

Moreover, $Q_1^1>0$ because  it was stated, e.g.\ in  \cite{MishuraRalchenkoSchilling2022}   inequalities (11), that for $H>1/2$
coefficients  $\rho_k$ are strictly decreasing, i.e.,
\begin{gather}\label{rhomonotest}
\rho_{k-1}>\rho_k>0,\;  \forall k\in\mathbb{N}.
\end{gather}

2) Let $n\ge 3$. On the one hand, it follows from equality  \eqref{projectj} that
\begin{gather}\label{Hnn-2}
H_n^{n-2}=\sum_{k=2, k\neq n}^{2n-2}Q_{n-2}^{|n-k|}\Delta_k.
\end{gather}
On the other hand, it follows from the tower property of conditional expectations and the inclusion of the corresponding     $\sigma$-algebras that
\begin{align}
 H_n^{n-2}
 &= \ex (\Delta_n\mid\Delta_{2},\ldots,\Delta_{n-1},\Delta_{n+1},\ldots,\Delta_{2n-2})\nonumber\\
 &=\ex (H_n^{n-1}\mid\Delta_{2},\ldots,\Delta_{n-1},\Delta_{n+1},\ldots,\Delta_{2n-2})\nonumber\\
 \label{Hnn-2M}
 &\stackrel{\mathclap{\eqref{projectn}}}{=}\;\sum_{k=2, k\neq n}^{2n-2}Q_{n-1}^{|n-k|}\Delta_k
 +Q_{n-1}^{n-1}(M_{n-2}^1+M_{n-2}^{2n-1}),
\end{align}
where
\begin{gather}
 M_{n-2}^1=\ex (\Delta_1\mid\Delta_{2},\ldots,\Delta_{n-1},\Delta_{n+1},\ldots,\Delta_{2n-2})=\sum_{k=2, k\neq n}^{2n-2}T_{n-2}^{k}\Delta_k,\nonumber\\
 \label{M2n-1}
 M_{n-2}^{2n-1}=\ex (\Delta_{2n-1}\mid\Delta_{2},\ldots,\Delta_{n-1},\Delta_{n+1},\ldots,\Delta_{2n-2})=\sum_{k=2, k\neq n}^{2n-2}\widetilde{T}_{n-2}^{k}\Delta_k,
\end{gather}
for some coefficients $T_{n-2}^{k},\ \widetilde{T}_{n-2}^{k}\in \real$ (we will define them during the proof).
Due to stationarity of the increments, $\widetilde{T}_{n-2}^{k}=T_{n-2}^{2n-k},\ 2\le k \le 2n-2$.

Again, due to tower property of conditional expectations,
\begin{equation}\begin{split}\label{towertwo}
  M_{n-2}^1
  &=\ex (\ex (\Delta_1\mid\Delta_{2},\ldots,\Delta_{n-1},\Delta_{n},\Delta_{n+1},\ldots,\Delta_{2n-2})\mid\Delta_{2},\ldots,\Delta_{n-1},\Delta_{n+1},\ldots,\Delta_{2n-2})\\
  &\stackrel{\mathclap{\eqref{projectgamma1}}}{=} \;\,
  \ex \left(\sum_{k=2}^{2n-2}\Gamma_{n}^k\Delta_k\biggm|\Delta_{2},\ldots,\Delta_{n-1},\Delta_{n+1},\ldots,\Delta_{2n-2}\right)\\
  &=\sum_{k=2, k\neq n}^{2n-2}\Gamma_{2n-2}^k\Delta_k
  +\Gamma_{2n-2}^nH_n^{n-2}\\
  &=\sum_{k=2, k\neq n}^{2n-2}\Gamma_{2n-2}^k\Delta_k
  +\Gamma_{2n-2}^n\sum_{k=2, k\neq n}^{2n-2}Q_{n-2}^{|n-k|}\Delta_k=\sum_{k=2, k\neq n}^{2n-2}T_{n-2}^{k}\Delta_k.
\end{split}\end{equation}

Equating the coefficients at $\Delta_k$ in the last equality in \eqref{towertwo}, we get \eqref{Tn-2k}.
Thus, taking to account \eqref{Hnn-2} and  \eqref{Hnn-2M}, we can conclude that
\begin{gather*}
 Q_{n-2}^{|n-k|}=Q_{n-1}^{|n-k|}
 +Q_{n-1}^{n-1}\left(T_{n-2}^{k}+T_{n-2}^{2n-k}\right),
 \quad 2\le k \le 2n-2,\ k\neq n,
\end{gather*}
which leads to \eqref{Qn-1k} via the relations
$$
|n-k|=|n-(2n-k)|,\quad
1\le |n-k|\le n-2,\ \mbox{ for }\ 2\le k \le 2n-2,\ k\neq n.
$$

Now we shall calculate  $Q_{n-1}^{n-1}$. For this purpose, let us present the projection $H_n^{n-1}$ in the following form:
\begin{align}
 H_n^{n-1}&=Q_{n-1}^{n-1}\Delta_1+\ldots+Q_{n-1}^{1}\Delta_{n-1}+Q_{n-1}^{1}\Delta_{n+1}+\ldots+Q_{n-1}^{n-1}\Delta_{2n-1}
 \nonumber\\
 \label{Hnn-1D1}
 &=Q_{n-1}^{n-1}I_1
 +\sum_{k=2, k\neq n}^{2n-1}\widetilde{Q}_{n-1}^k\Delta_k.
\end{align}
Here $\widetilde{Q}_{n-1}^k\in \real$, $2\le k \le 2n-1$, $k\neq n$ are some coefficients, and the random variable $I_1$ equals
$$I_1=\Delta_1-\ex \left(\Delta_1\mid\Delta_{2},\ldots,\Delta_{n-1},\Delta_{n+1},\ldots,\Delta_{2n-1}\right).
$$

Multiply both sides of \eqref{Hnn-1D1} by
$I_1$ and take the expectations. Since
\begin{gather}\label{CovI1Dk}
\Cov {\left(I_1,\Delta_k\right)}
=\ex \Delta_1\Delta_k
-\ex \left(\Delta_k\ex \left(\Delta_1\mid\Delta_{2},\ldots,\Delta_{n-1},\Delta_{n+1},\ldots,\Delta_{2n-1}\right)\right)=0,
\end{gather}
for all $2\le k \le 2n-1, k\neq n$,
we get
$$
\ex \left(I_1\sum_{k=2, k\neq n}^{2n-1}\widetilde{Q}_{n-1}^k\Delta_k\right)=0.
$$

Moreover, due to the fact that
$$
\ex (H_n^{n-1} I_1)
=\ex \left(\ex \left(\Delta_n I_1\mid\Delta_{1},\ldots,\Delta_{n-1},\Delta_{n+1},\ldots,\Delta_{2n-1}\right)\right)
=\ex (\Delta_n I_1),
$$
we obtain from \eqref{Hnn-1D1}:
\begin{gather}\label{EDnI1}
 \ex (\Delta_n I_1)=Q_{n-1}^{n-1}\ex (I_1)^2.
\end{gather}

Thus,
\begin{gather}\label{Qn-1n-1ED}
 Q_{n-1}^{n-1}
 =\frac{\ex (\Delta_n \Delta_1-\Delta_n\ex \left(\Delta_1\mid\Delta_{2},\ldots,\Delta_{n-1},\Delta_{n+1},\ldots,\Delta_{2n-1}\right))}{\ex (\Delta_1-\ex \left(\Delta_1\mid\Delta_{2},\ldots,\Delta_{n-1},\Delta_{n+1},\ldots,\Delta_{2n-1}\right))^2},
\end{gather}
and our next step is to determine $\ex \left(\Delta_1\mid\Delta_{2},\ldots,\Delta_{n-1},\Delta_{n+1},\ldots,\Delta_{2n-1}\right)$.

By \eqref{projectgamma1},
\begin{gather*}
 \ex \left(\Delta_1\mid\Delta_{2},\ldots,\Delta_{n-1},\Delta_{n},\Delta_{n+1},\ldots,\Delta_{2n-1}\right)
 =\sum_{k=2}^{2n-1}\Gamma_{2n-1}^k\Delta_k.
\end{gather*}
Hence, similarly to the calculation of $M_{n-2}^1$ in \eqref{towertwo}, we obtain
\begin{align}
\MoveEqLeft
 \ex \left(\Delta_1\mid\Delta_{2},\ldots,\Delta_{n-1},\Delta_{n+1},\ldots,\Delta_{2n-1}\right)\nonumber\\
 \label{ED1n}
 &=\sum_{k=2, k\neq n}^{2n-1}\Gamma_{2n-1}^k\Delta_k
 +\Gamma_{2n-1}^n \ex \left(\Delta_n\mid\Delta_{2},\ldots,\Delta_{n-1},\Delta_{n+1},\ldots,\Delta_{2n-1}\right),
\end{align}
and now we need to determine the last conditional expectation in \eqref{ED1n}.

Using the same reasoning as in \eqref{Hnn-1D1}, we get
\begin{align}
\MoveEqLeft \ex \left(\Delta_n\mid\Delta_{2},\ldots,\Delta_{n-1},\Delta_{n+1},\ldots,\Delta_{2n-1}\right)
 =\sum_{k=2, k\neq n}^{2n-1}S_{2n-1}^k\Delta_k\nonumber\\
 \label{ED1S}
 &=S_{2n-1}^{2n-1}J_{2n-1}
 +\sum_{k=2, k\neq n}^{2n-2}\widetilde{S}_{2n-1}^k\Delta_k,
\end{align}
where $\widetilde{S}_{2n-1}^k\in\real$, $2\le k \le 2n-2$, $k\neq n$ are some coefficients, and
$$J_{2n-1}=\Delta_{2n-1}
 - \ex \left(\Delta_{2n-1}\mid\Delta_{2},\ldots,\Delta_{n-1},\Delta_{n+1},\ldots,\Delta_{2n-2}\right)
 =\Delta_{2n-1}-M_{n-2}^{2n-1}.
$$
Multiplying both sides of \eqref{ED1S} by
$J_{2n-1}$ and taking the expectations, we arrive at
\begin{gather}\label{EDJ}
 \ex (\Delta_n J_{2n-1})=S_{2n-1}^{2n-1} \ex (J_{2n-1})^2.
\end{gather}

To continue, we substitute \eqref{M2n-1} in \eqref{EDJ} and  obtain the equalities
\begin{gather*}
 S_{2n-1}^{2n-1}=\frac{\ex (\Delta_n\Delta_{2n-1}-\Delta_n\sum_{k=2, k\neq n}^{2n-2}T_{n-2}^{2n-k}\Delta_k)}{\ex (J_{2n-1})^2}
 =\frac{\rho_{n-1}-\sum_{k=2, k\neq n}^{2n-2}T_{n-2}^{2n-k}\rho_{|n-k|}}{1-\sum_{k=2, k\neq n}^{2n-2}T_{n-2}^{2n-k}\rho_{|2n-1-k|}}.
\end{gather*}
With the same reasonings that allowed us to  get \eqref{EDnI1}, we can use  the orthogonality of the random variables $\Delta_k$ and $J_{2n-1}$ for all $2\le k \le 2n-1$, $k\neq n$.  With this at hand, we get that
\begin{align*}
 \ex (J_{2n-1})^2
 &=\ex \left[
 \bigg(\Delta_{2n-1}
 -\sum_{k=2, k\neq n}^{2n-2}T_{n-2}^{2n-k}\Delta_k
 \bigg)^2
 \right]\\
 &=\ex \left[\Delta_{2n-1}
 \bigg(\Delta_{2n-1}
 -\sum_{k=2, k\neq n}^{2n-2}T_{n-2}^{2n-k}\Delta_k
 \bigg)\right]\\
 &\quad- \sum_{k=2, k\neq n}^{2n-2}T_{n-2}^{2n-k}\ex \left[\Delta_k
 \bigg(\Delta_{2n-1}
 -\sum_{k=2, k\neq n}^{2n-2}T_{n-2}^{2n-k}\Delta_k
 \bigg)\right],
\end{align*}
where the last term equals to 0. Thus, we have proved equality \eqref{S2n-1}.

Further, let us determine coefficients $S_{2n-1}^{k}$, $2\le k\le 2n-2$ from \eqref{ED1S}.
It follows from  the tower property of conditional expectations  and the first equality of \eqref{ED1S} that
\begin{align*}
 H_n^{n-2}
 &=\ex \left(\ex \left(\Delta_n\mid\Delta_{2},\ldots,\Delta_{n-1},\Delta_{n+1},\ldots,\Delta_{2n-1}\right)\mid\Delta_{2},\ldots,\Delta_{n-1},\Delta_{n+1},\ldots,\Delta_{2n-2}\right)
 \\
 &=\ex \left(\sum_{k=2, k\neq n}^{2n-1}S_{2n-1}^k\Delta_k\biggm|\Delta_{2},\ldots,\Delta_{n-1},\Delta_{n+1},\ldots,\Delta_{2n-2}\right)\\
 &=\sum_{k=2, k\neq n}^{2n-2}S_{2n-1}^k\Delta_k+S_{2n-1}^{2n-1}M_{n-2}^{2n-1}\\
 &\stackrel{\mathclap{\eqref{M2n-1}}}{=}\;
 \sum_{k=2, k\neq n}^{2n-2}S_{2n-1}^k\Delta_k
 +\sum_{k=2, k\neq n}^{2n-2}S_{2n-1}^{2n-1}T_{n-2}^{2n-k}\Delta_k
 \stackrel{\eqref{Hnn-2}}{=}\sum_{k=2, k\neq n}^{2n-2}Q_{n-2}^{|n-k|}\Delta_k.
\end{align*}

Equating the coefficients at $\Delta_k,\ 2\le k \le 2n-2, k\neq n,$ in the last relation, we establish \eqref{Sk}.

Consiquently, we have that
\begin{align}
\MoveEqLeft
 \ex \left(\Delta_1\mid\Delta_{2},\ldots,\Delta_{n-1},\Delta_{n+1},\ldots,\Delta_{2n-1}\right)\nonumber\\
 \label{ED1G}
 &=\sum_{k=2, k\neq n}^{2n-1}\Gamma_{2n-1}^k\Delta_k
 +\Gamma_{2n-1}^n \sum_{k=2, k\neq n}^{2n-1}S_{2n-1}^k\Delta_k
 =\sum_{k=2, k\neq n}^{2n-1}G_{2n-1}^k\Delta_k,
\end{align}
where coefficients $G_{2n-1}^k$ defined by \eqref{G2n-1}.

Substituting \eqref{ED1G} into \eqref{Qn-1n-1ED} and applying of  \eqref{CovI1Dk} leads us to \eqref{Qn-1n-1}, which completes the proof of the theorem.
\end{proof}

\section{Calculation of bilateral projection coefficients for  \texorpdfstring{$\bm{n=2,3,4}$}{n=2,3,4}}\label{sec4}
%%%%%%%%%%% Partial cases n=2,3,4
It is obvious that in the general case Theorem \ref{theor2.1} gives formulas for calculating projection coefficients in a rather complicated form, and the possibility of simplifying them does not seem realistic.  Therefore it is interesting to consider particular cases, when the formulas  for  the projection coefficients have comparatively simple and observable form. So, let us look at some special cases, namely, consider subsequently the values  $n=2,3,4$, and calculate projection coefficients $Q_{n-1}^k$ from \eqref{projectn}, using system~\eqref{systq} and respective matrix equation~\eqref{vecteq}. To understand the behaviour of coefficients as the functions of $H$, we recall that in the ultimate case $H=1$ fractional Brownian motion $B_t^1=\xi t$, $t\ge 0$, where $\xi$ is a standard normal variable, whence all $\Delta_k$ equals $\xi$ and consequently all $\rho_k$ equal 1.

To solve \eqref{systq} we use Cramer's rule, namely,
\begin{gather}\label{QCramer}
 Q_{n-1}^{k}=\frac{D_{n-1,k}^*}{D_{n-1}^*}, %=\frac{\det(A_{n-1,k}^*)}{\det(A_{n-1}^*)}
\end{gather}
where $D_{n-1,k}^*=\det(A_{n-1,k}^*),\ D_{n-1}^*=\det(A_{n-1}^*)$, matrix $A_{n-1}^*$ is defined by \eqref{matrA*n-1} and $A_{n-1,k}^*$ is the matrix $A_{n-1}^*$ with its $(n-k)$th column vector replaced by $\overline{\rho}_{n-1}^*$ from equation~\eqref{vecteq}. We illustrate each case with the corresponding graphs.

\subsection{Case \texorpdfstring{$\bm{n=2}$}{n=2}} Recall that according to item 1) from Theorem \ref{theor2.1}, in this case
we have the unique coefficient $Q_1^1$, and it equals
\begin{gather*}
 Q_1^1=\frac{\rho_{1}}{1+\rho_2}.
\end{gather*}
As expected, we see on the Figure \ref{fig:Q1} that the coefficient $ Q_1^1$ increases from zero to $\frac12$ as $H$ increases from $\frac12$ to 1. Moreover, it is a convex upward function of $H$.

\begin{figure}[ht!]
 \includegraphics[scale=0.75]{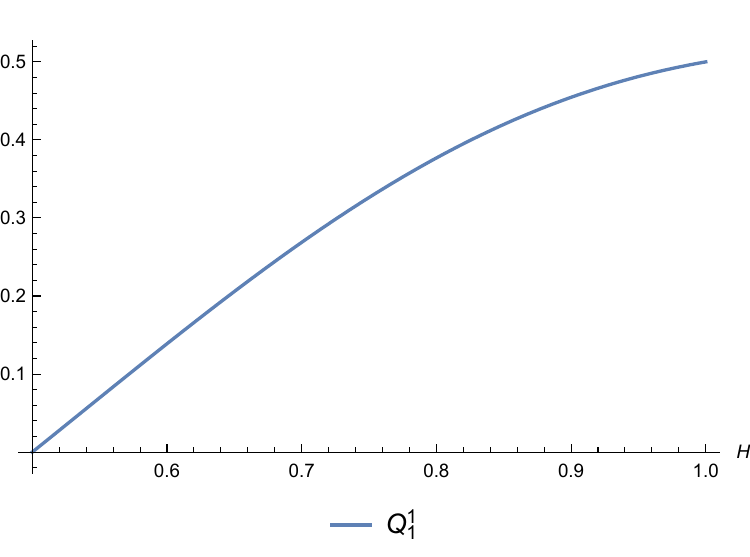}
 \caption{Graph of $Q_1^1$ as a function of $H$.}\label{fig:Q1}
\end{figure}

\subsection{Case \texorpdfstring{$\bm{n=3}$}{n=3}} In this case the determinants of the system \eqref{systq} equal
\begin{gather*}
 D_{2}^*=(1+\rho_2)(1+\rho_4)-(\rho_1+\rho_3)^2,\\
 D_{2,1}^*=\rho_1(1+\rho_4)-\rho_2(\rho_1+\rho_3),\quad
 D_{2,2}^*=\rho_2(1+\rho_2)-\rho_1(\rho_1+\rho_3),
\end{gather*}

and we obtain the following values of the coefficients:
\begin{gather*}
Q_2^1=\dfrac{\rho_1(1+\rho_4)-\rho_2(\rho_1+\rho_3)}{(1+\rho_2)(1+\rho_4)-(\rho_1+\rho_3)^2},
\quad
Q_2^2=\dfrac{\rho_2(1+\rho_2)-\rho_1(\rho_1+\rho_3)}{(1+\rho_2)(1+\rho_4)-(\rho_1+\rho_3)^2}.
\end{gather*}

Let us show that coefficients $Q_2^1,\ Q_2^2$ are strictly positive for all $H\in(1/2,1)$.

By \eqref{rhomonotest}, $1 > \rho_1 > \rho_2 > \rho_3$. Thus,
\begin{gather}\label{rhoest1}
 1+\rho_2>\rho_1+\rho_3.
\end{gather}

Also, due to inequality  (12) from \cite{MishuraRalchenkoSchilling2022}, $\rho_{k-1}-\rho_{k}>\rho_{k}-\rho_{k+1}$, whence
\begin{gather}\label{rhoest2}
 1-\rho_1>\rho_{3}-\rho_4 \quad\Leftrightarrow\quad
1+\rho_4>\rho_{1}+\rho_3.
\end{gather}

Taking \eqref{rhoest1} and the right-hand side of \eqref{rhoest2} into account,  we immediately get that $D_{2}^*>0$.
Moreover,  taking again the right-hand side of  inequality~\eqref{rhoest2} and relation $\rho_1>\rho_2$ into account, we obtain that $$\rho_1(1+\rho_4)>\rho_2(\rho_{1}+\rho_3),$$ i.e., $D_{2,1}^*>0$  and consequently $Q_2^1>0$.
Furthermore, by inequality~(24) from \cite{MishuraRalchenkoSchilling2022},
$\rho_2(1+\rho_2)>\rho_1(\rho_1+\rho_3)$. Hence, $D_{2,2}^*>0$, and $Q_2^2>0$.
As we can see from Figure \ref{fig:Q2}, coefficient $Q_2^1$ is strictly increasing and convex upward function of $H$, while $Q_2^2$ is also convex upward but changes its monotonicity. The maximum value of $Q_2^2$ is attained at the point $H \approx 0.7807$ and equals $0.0733648$.
The limits of $Q_2^1$ and $Q_2^2$ as $H\uparrow1$ are equal to $0.459546$ and $0.040454$ respectively, and the sum of the limits equals $\frac12.$
\begin{figure}[ht!]
 \includegraphics[scale=0.75]{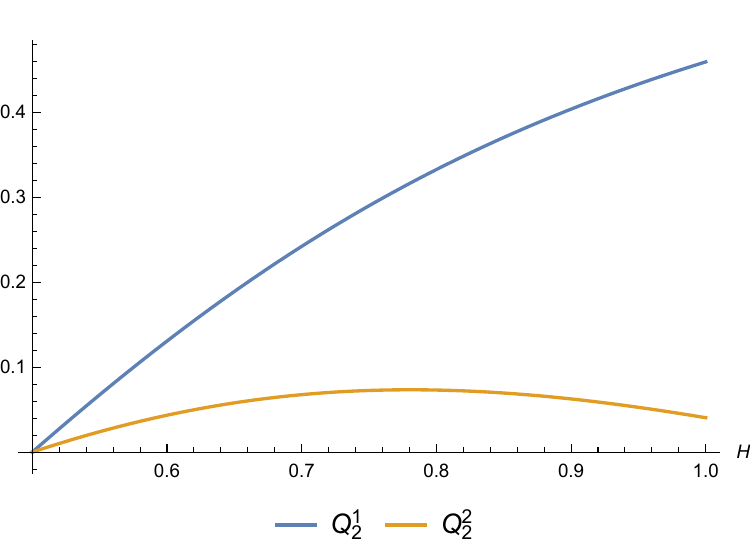}
 \caption{Graphs of $Q_2^1$ and $Q_2^2$ as functions of $H$.}\label{fig:Q2}
\end{figure}

\subsection{Case \texorpdfstring{$\bm{n=4}$}{n=4}}
In this case the covariance matrix has a form
\begin{gather*}
 A_{3}^*=
 \begin{pmatrix}
1+\rho_{6} & \rho_{1}+\rho_{5} & \rho_{2}+\rho_{4}\\
\rho_{1}+\rho_{5} & 1+\rho_{4} & \rho_{1}+\rho_{3}\\
\rho_{2}+\rho_{4} & \rho_{1}+\rho_{3} & 1+\rho_{2}
\end{pmatrix}.
\end{gather*}

Thus,
\begin{align*}
 D_{3}^*&=(1+\rho_6)(1+\rho_4)(1+\rho_2)-(1+\rho_6)(\rho_1+\rho_3)^2-(1+\rho_4)(\rho_2+\rho_4)^2-(1+\rho_2)(\rho_1+\rho_5)^2\\*
 &\quad+2(\rho_1+\rho_3)(\rho_2+\rho_4)(\rho_1+\rho_5),\\
  D_{3,1}^*&=(1+\rho_6)
  \big(\rho_1(1+\rho_4)-\rho_2(\rho_1+\rho_3)\big)
  -(\rho_1+\rho_5)
  \big(\rho_1(\rho_1+\rho_5)-\rho_3(\rho_1+\rho_3)\big)\\
  &\quad+(\rho_2+\rho_4)
  \big(\rho_2(\rho_1+\rho_5)-\rho_3(1+\rho_4)\big),\\
   D_{3,2}^*&=(1+\rho_6)
  \big(\rho_2(1+\rho_2)-\rho_1(\rho_1+\rho_3)\big)
  -(\rho_1+\rho_5)
  \big(\rho_3(1+\rho_2)-\rho_1(\rho_2+\rho_4)\big)\\
  &\quad+(\rho_2+\rho_4)
  \big(\rho_3(\rho_1+\rho_3)-\rho_2(\rho_2+\rho_4)\big),\\
   D_{3,3}^*&=\rho_3
  \big((1+\rho_4)(1+\rho_2)-(\rho_1+\rho_3)^2\big)
  -(\rho_1+\rho_5)
  \big(\rho_2(1+\rho_2)-\rho_1(\rho_1+\rho_3)\big)\\
  &\quad+(\rho_2+\rho_4)
  \big(\rho_2(\rho_1+\rho_3)-\rho_1(1+\rho_4)\big),
\end{align*}
and coefficients $Q_{3}^k$, $1\le k \le 3$, can be calculated by formula \eqref{QCramer}.

 It was proved in \cite[Theorem~1.1]{bamish}  that the covariance matrices of fractional Brownian motion and fractional Gaussian noise are non-degenerate. Therefore, $D_{3}^*>0$.
Let us  try to investigate the sign of $D_{3,1}^*$.
Obviously, $D_{3,1}^*$ starts and finishes with zero values. Furthermore, while the analysis of the determinant  $D_{3,1}^*$ is difficult, and simple comparisons of the coefficients, as was done for $n=2$, seems to be impossible, it is very easy to analyze it numerically, and the result is presented at Figure \ref{fig:D31}. We see that $D_{3,1}^*$ is strictly positive between $H=\frac12$ and $H=1$.
The determinant $D^*_{3,3}$ can be analyzed numerically in a similar manner. We mention that it also remains positive for $\frac12 < H < 1$, see Figure \ref{fig:D33}.

\begin{figure}[ht]
\centering
\subcaptionbox{$D_{3,1}^*$\label{fig:D31}}
{\includegraphics[scale=0.6]{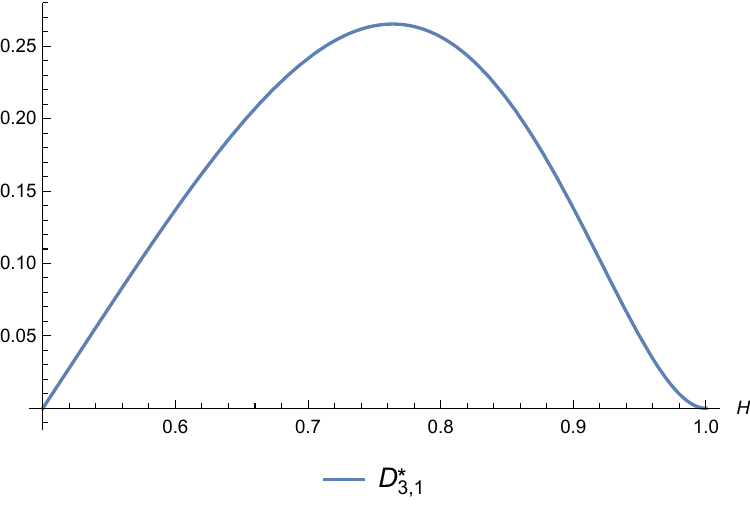}}%
\subcaptionbox{$D_{3,3}^*$\label{fig:D33}}{\includegraphics[scale=0.6]{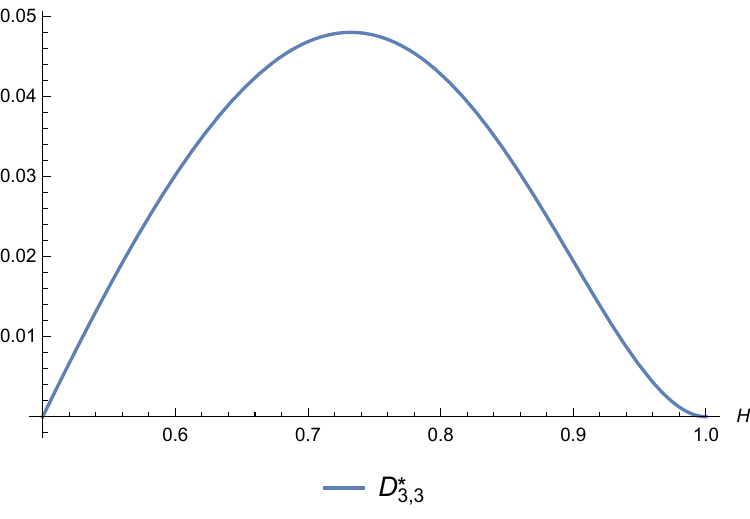}}
 \caption{Graphs of $D_{3,1}^*$ and $D_{3,3}^*$ as functions of $H$}\label{fig:D31_D33}
\end{figure}

Now, let us investigate the determinant $D_{3,2}^*$. Surprisingly, Figure \ref{fig:D32} shows that $D_{3,2}^*$ becomes negative for values of
$H$ close to one, starting from approximately 0.99300. However, the corresponding negative values of $D_{3,2}^*$ have extremely small absolute values: the minimum value is
 $-4.844\cdot10^{-7}$.
To ensure this phenomenon is not caused by numerical calculation errors, we will analyze the behavior of $D_{3,2}^*$  as a function of
$H$ in more detail. The next proposition analytically shows that $D_{3,2}^*$ is indeed negative in a left neighborhood of~1.

\begin{figure}
\centering
\subcaptionbox{$H\in(0.5,1)$}
{\includegraphics[scale=0.59]{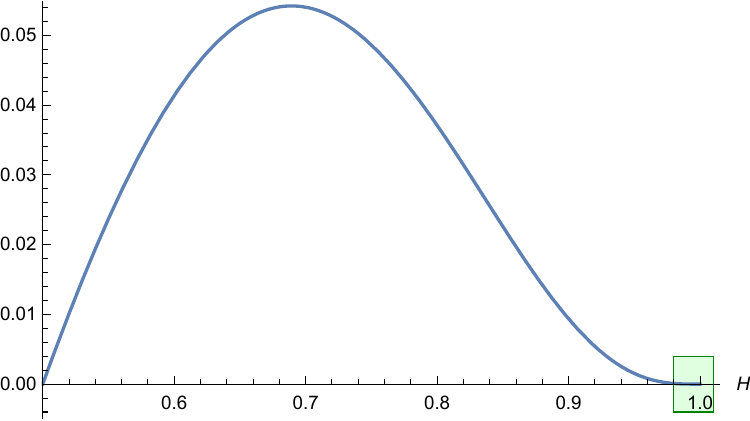}}%
\ \ 
\subcaptionbox{$H\in(0.99,1)$}{\includegraphics[scale=0.64]{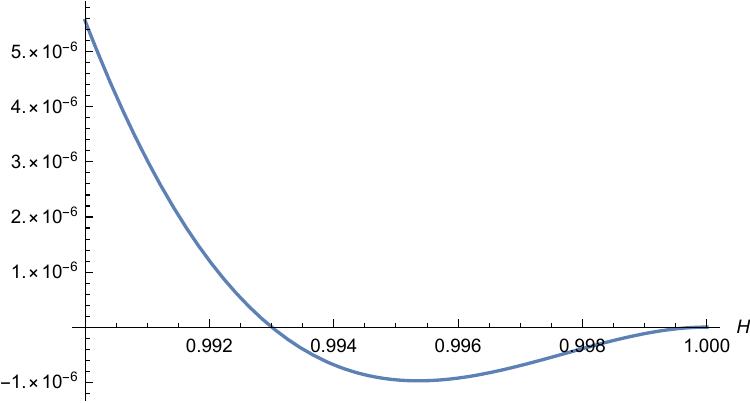}}
 \caption{Graph of $D_{3,2}^*$ as a function of $H$}\label{fig:D32}
\end{figure}

\begin{proposition}
There exists some $\delta\in(0,1)$ such that $D_{3,2}^* < 0$ for all $H\in(\delta,1)$.
\end{proposition}

\begin{proof}
The proof is based on analyzing the first and second derivatives of $D_{3,2}^*$ at $H=1$.
We shall show that
\begin{equation}\label{eq:D32deriv}
\lim_{H\uparrow1} D_{3,2}^* = 0,\quad
\lim_{H\uparrow1} \partial_H D_{3,2}^* = 0, \quad
\lim_{H\uparrow1}  \partial_{H\!H} D_{3,2}^* < 0.
\end{equation}
as illustrated in Figures \ref{fig:D32}--\ref{fig:der2D32}, which shows the graphs of $D_{3,2}^*$ and its derivatives.
These results collectively imply the desired conclusion. Specifically, the final inequality indicates that the first derivative $\partial_H D_{3,2}^*$ strictly decreases as $H$ approaches 1. Given that $\partial_H D_{3,2}^*$ is continuous and equals zero at $H = 1$, it follows that
$\partial_H D_{3,2}^*$ is positive on some interval
$(\delta,1)$. Consequently, the determinant $D_{3,2}^*$ itself strictly increases on $(\delta,1)$. Since $D_{3,2}^*=0$ at $H = 1$, we can conclude that
$D_{3,2}^* < 0$ for $H \in (\delta,1)$.

\begin{figure}
\centering
%\subcaptionbox{$H\in(0.5,1)$}
%{\includegraphics[scale=0.6]{plot_dD32}}
%\quad\subcaptionbox{$H\in(0.99,1)$}
%{\includegraphics[scale=0.65]{plot_dD32b}}\\[15pt]
\includegraphics[scale=0.83]{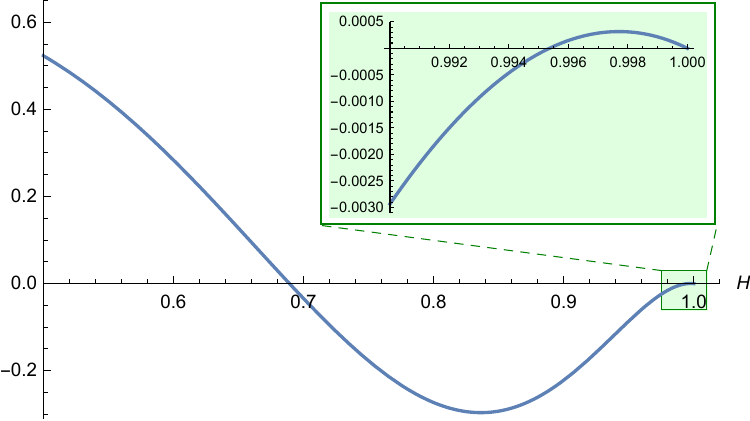}
\caption{The derivative $\partial_H D_{3,2}^*$ as a function of $H$}
\end{figure}

\begin{figure}
\centering
\includegraphics[scale=0.83]{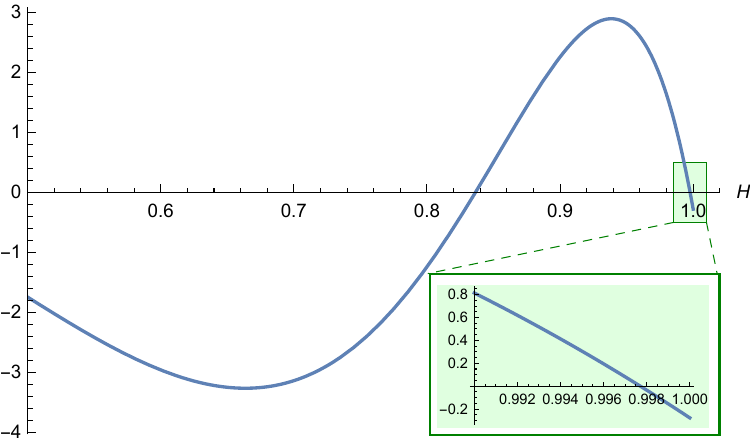}
\caption{The second derivative $\partial_{H\!H} D_{3,2}^*$ as a function of $H$}\label{fig:der2D32}
\end{figure}

The first equality in \eqref{eq:D32deriv} follows from the fact that if $H=1$, then $\rho_k = 1$ for any $k$. Therefore,
\[
 D_{3,2}^* =
 \begin{vmatrix}
1+\rho_6 & \rho_3 & \rho_2+\rho_4\\
\rho_1+\rho_5 & \rho_2 & \rho_1+\rho_3\\
\rho_2+\rho_4 & \rho_1 & 1+\rho_2
\end{vmatrix}
\xrightarrow[H\uparrow1]{}
 \begin{vmatrix}
2 & 1 & 2\\
2 & 1 & 2\\
2 & 1 & 2
\end{vmatrix}
=0.
\]
Similarly, for the first derivative we have (denoting by $\hat\rho_k'$ the values of $\rho_k'$ at $H=1$)
\begin{align*}
\partial_H D_{3,2}^* &=
 \begin{vmatrix}
\rho_6' & \rho_3 & \rho_2+\rho_4\\
\rho_1'+\rho_5' & \rho_2 & \rho_1+\rho_3\\
\rho_2'+\rho_4' & \rho_1 & 1+\rho_2
\end{vmatrix}
+
\begin{vmatrix}
1+\rho_6 & \rho_3' & \rho_2+\rho_4\\
\rho_1+\rho_5 & \rho_2' & \rho_1+\rho_3\\
\rho_2+\rho_4 & \rho_1' & 1+\rho_2
\end{vmatrix}
+
\begin{vmatrix}
1+\rho_6 & \rho_3 & \rho_2'+\rho_4'\\
\rho_1+\rho_5 & \rho_2 & \rho_1'+\rho_3'\\
\rho_2+\rho_4 & \rho_1 & \rho_2'
\end{vmatrix}
\\
&\xrightarrow[H\uparrow1]{}
\begin{vmatrix}
\hat\rho_6' & 1 & 2\\
\hat\rho_1'+\hat\rho_5' & 1 & 2\\
\hat\rho_2'+\hat\rho_4' & 1 & 2
\end{vmatrix}
+
\begin{vmatrix}
2 & \hat\rho_3' & 2\\
2 & \hat\rho_2' & 2\\
2 & \hat\rho_1' & 2
\end{vmatrix}
+
\begin{vmatrix}
2 & 1 & \hat\rho_2'+\hat\rho_4'\\
2 & 1 & \hat\rho_1'+\hat\rho_3'\\
2 & 1 & \hat\rho_2'
\end{vmatrix}
= 0,
\end{align*}
because all matrices in the limit have two linearly dependent columns.

Finally, let us calculate the second derivative. We have
\begin{align*}
\partial_{H\!H} D_{3,2}^* &=
\begin{vmatrix}
\rho_6'' & \rho_3 & \rho_2+\rho_4\\
\rho_1''+\rho_5' & \rho_2 & \rho_1+\rho_3\\
\rho_2''+\rho_4' & \rho_1 & 1+\rho_2
\end{vmatrix}
+
\begin{vmatrix}
1+\rho_6 & \rho_3'' & \rho_2+\rho_4\\
\rho_1+\rho_5 & \rho_2'' & \rho_1+\rho_3\\
\rho_2+\rho_4 & \rho_1'' & 1+\rho_2
\end{vmatrix}
+
\begin{vmatrix}
1+\rho_6 & \rho_3 & \rho_2''+\rho_4''\\
\rho_1+\rho_5 & \rho_2 & \rho_1''+\rho_3''\\
\rho_2+\rho_4 & \rho_1 & \rho_2''
\end{vmatrix}
\\
&\quad+ 2
\begin{vmatrix}
\rho_6' & \rho_3' & \rho_2+\rho_4\\
\rho_1'+\rho_5' & \rho_2' & \rho_1+\rho_3\\
\rho_2'+\rho_4' & \rho_1' & 1+\rho_2
\end{vmatrix}
+ 2
\begin{vmatrix}
\rho_6' & \rho_3 & \rho_2'+\rho_4'\\
\rho_1'+\rho_5' & \rho_2 & \rho_1'+\rho_3'\\
\rho_2'+\rho_4' & \rho_1 & \rho_2'
\end{vmatrix}
+ 2
\begin{vmatrix}
1+\rho_6 & \rho_3' & \rho_2'+\rho_4'\\
\rho_1+\rho_5 & \rho_2' & \rho_1'+\rho_3'\\
\rho_2+\rho_4 & \rho_1' & \rho_2'
\end{vmatrix}
\\
&\xrightarrow[H\uparrow1]{}
2
\begin{vmatrix}
\hat\rho_6' & \hat\rho_3' & 2\\
\hat\rho_1'+\hat\rho_5' & \hat\rho_2' & 2\\
\hat\rho_2'+\hat\rho_4' & \hat\rho_1' & 2
\end{vmatrix}
+ 2
\begin{vmatrix}
\hat\rho_6' & 1 & \hat\rho_2'+\hat\rho_4'\\
\hat\rho_1'+\hat\rho_5' & 1 & \hat\rho_1'+\hat\rho_3'\\
\hat\rho_2'+\hat\rho_4' & 1 & \hat\rho_2'
\end{vmatrix}
+ 2
\begin{vmatrix}
2 & \hat\rho_3' & \hat\rho_2'+\hat\rho_4'\\
2 & \hat\rho_2' & \hat\rho_1'+\hat\rho_3'\\
2 & \hat\rho_1' & \hat\rho_2'
\end{vmatrix}.
\end{align*}
Expanding the determinants and rearranging terms, we arrive at
\begin{equation}\label{eq:secder}
\begin{split}
\lim_{H\uparrow1}\partial_{H\!H} D_{3,2}^* &=
6 \hat\rho_1' \hat\rho_2' - 4 \hat\rho_1' \hat\rho_3'
+ 8 \hat\rho_1' \hat\rho_4' + 4 \hat\rho_1' \hat\rho_5'
- 6\hat\rho_1' \hat\rho_6' - 6(\hat\rho_2')^2
+ 2 \hat\rho_2' \hat\rho_3' - 12 \hat\rho_2' \hat\rho_4'
\\
&\quad + 6 \hat\rho_2' \hat\rho_6' + 4 (\hat\rho_3')^2
+ 6 \hat\rho_3' \hat\rho_4' - 4 \hat\rho_3' \hat\rho_5'
- 2 \hat\rho_3' \hat\rho_6' - 2 (\hat\rho_4')^2
+ 2 \hat\rho_4' \hat\rho_5'.
\end{split}
\end{equation}
Next, we calculate the values of
$\hat\rho_1', \dots, \hat\rho_6'$:
\begin{align*}
\hat\rho_1' &= 4 \log 2,
&
\hat\rho_2' &= 9 \log 3 - 8 \log 2,
\\
\hat\rho_3' &= 36 \log 2 - 18 \log 3,
&
\hat\rho_4' &= -64 \log 2 + 9 \log 3 + 25 \log 5,
\\
\hat\rho_5' &= 32 \log 2 - 50 \log 5 + 72 \log 6,
&
\hat\rho_6' &= 25 \log 5 -72 \log 6 + 49 \log 7.
\end{align*}
After substituting these values into \eqref{eq:secder} and computing the numerical value, we find that
\[
\lim_{H\uparrow1}\partial_{H\!H} D_{3,2}^* \approx -0.277226.
\]
Thus, \eqref{eq:D32deriv} holds, and the proof follows.
\end{proof}

Concerning the projection coefficients, they are presented in Figure \ref{fig:Q3}. The maximum value of $Q_3^2$ is attained at $H \approx 0.7152$ and equals $0.0530381$.
The maximum value of $Q_3^3$ is attained at $H \approx 0.8729$ and equals $0.0554454$.
The limits of $Q_3^1$, $Q_3^2$, and $Q_3^3$ as $H\uparrow1$ are $0.449901$, $-0.002201$ and $0.052300$ respectively.
Note that the sum of the coefficients still converges to $0.5$.

\begin{figure}[ht]
 \includegraphics[scale=0.65]{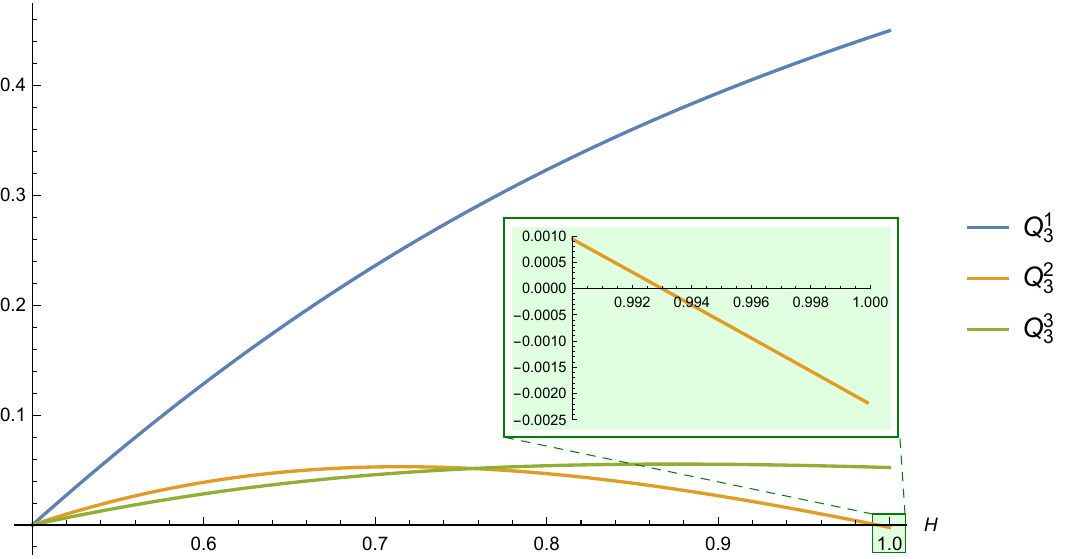}
 \caption{Graphs of $Q_3^1$, $Q_3^2$ and $Q_3^3$ as functions of $H$.}\label{fig:Q3}
\end{figure}

\subsection{Further numerical results}
For the cases $n = 5$ and $n = 6$, we provide the graphs in Figures \ref{fig:Q4} and \ref{fig:Q5}. In both cases, all coefficients except the second ones (i.e., $Q_4^2$ and $Q_5^2$) are positive for all $H \in (0.5, 1)$. The second coefficients become negative as $H$ approaches 1.

\begin{figure}[ht]
 \includegraphics[scale=0.65]{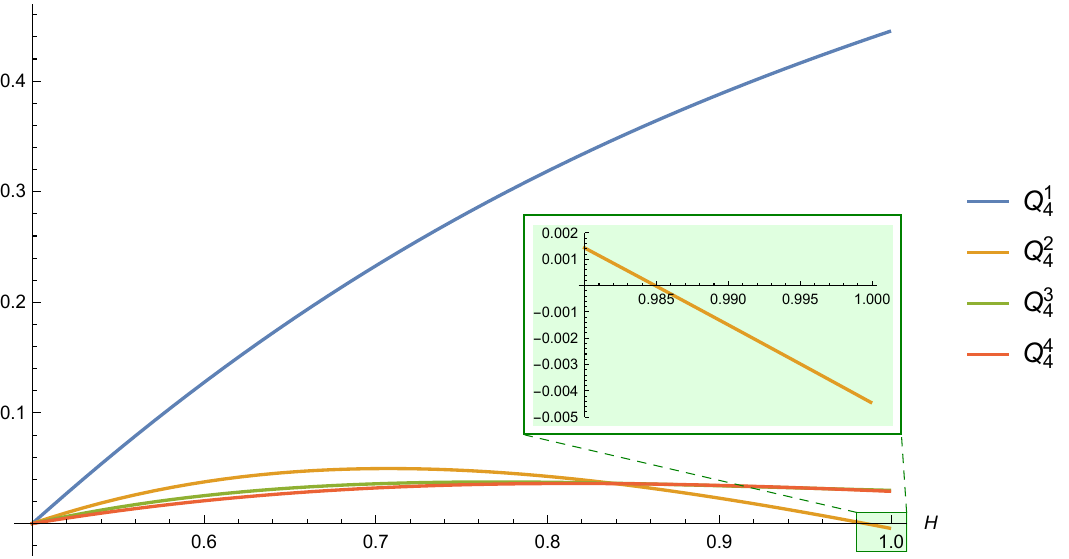}
 \caption{Graphs of $Q_4^k$, $k = 1, 2, 3, 4$, as functions of $H$.}\label{fig:Q4}
\end{figure}

\begin{figure}[ht]
 \includegraphics[scale=0.8]{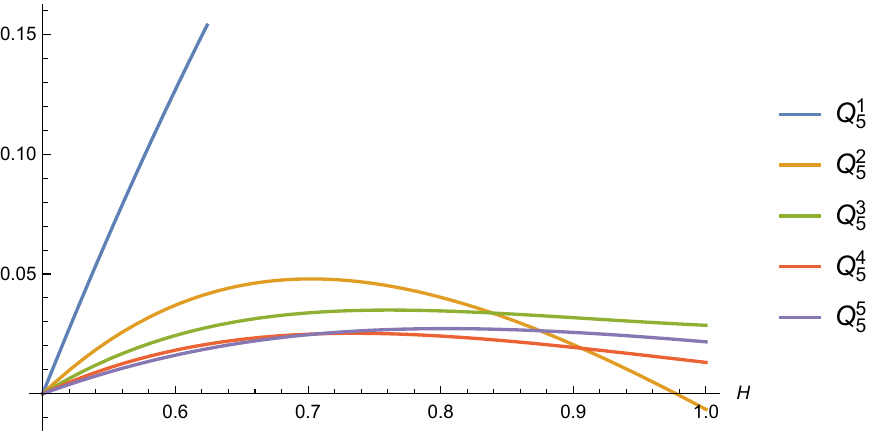}
 \caption{Graphs of $Q_5^k$, $k = 1, \dots, 5$, as functions of $H$.}\label{fig:Q5}

\end{figure}

Figure \ref{fig:Qsecond} illustrates the behavior of the second coefficient $Q_n^2$ for various $n = 2, \dots, 6$ as a function of $H \in (0.5, 1)$. All graphs start at zero, increase to a maximum point, and then decrease. Except for $Q_2^2$, all of them eventually become negative.

\begin{figure}[th]
 \includegraphics[scale=.85]{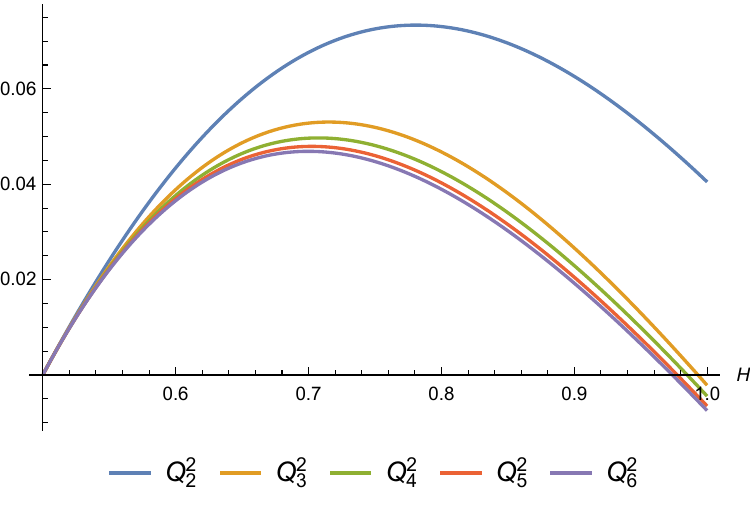}
  \qquad
    \begin{tabular}[b]{cc}\toprule
      $n$ & Root of $Q_n^2$ \\ \midrule
      3 & 0.99300 \\
      4 & 0.98491 \\
      5 & 0.97742 \\
      6 & 0.97375\\
      7 & 0.97143\\
      8 & 0.96991\\
      9 & 0.96885\\
      10 & 0.96807\\
      \bottomrule
      \\[3mm]
      \\
      \\
      \\
    \end{tabular}
 \caption{Graphs of $Q_n^2$, $n = 2, \dots, 6$, as functions of $H$. Points of intersection with $H$-axis.}\label{fig:Qsecond}
\end{figure}

We also compute the coefficients $Q_n^k$ for various values of $H$ numerically. Tables \ref{tab:0.51}--\ref{tab:0.99} list the results for $H = 0.51$, $0.6$, $0.7$, $0.8$, $0.9$, and $0.99$ for $1 \le n \le 10$.

Let us summarize our findings from the graphical and numerical results. The following patterns can be observed:
\begin{enumerate}[\itshape(i)]
\item All coefficients are positive except $Q_n^2$ for $n \ge 3$.

\item \label{o2} The first coefficient in each row is the largest, i.e., $Q_n^1 > Q_n^k$ for any $2 \le k \le n$. Often, it is substantially larger than any other coefficient in the row.

\item \label{o3} Monotonicity along each column holds, i.e., $Q_n^k > Q_{n+1}^k$ for fixed $k$.

\item As a function of $H$, the first coefficient in each row, i.e., $Q_n^1$ for $n \ge 1$, increases as $H$ increases from 0.5 to 1. Other coefficients ($Q_n^k$ for $2\le k\le n$) increase to a maximum point and then decrease.

\item All coefficients are convex upwards functions of $H \in (0.5, 1)$.

\item For each fixed $n$, the sum of the coefficients $Q_n^1 + \dots + Q_n^n$ converges to 0.5 as $H \uparrow 1$.

\end{enumerate}

Note that properties \ref{o2} and \ref{o3} are similar to those of the coefficients $\Gamma_n^k$, as studied in \cite{MishuraRalchenkoSchilling2022}. Unlike $Q_n^k$, all $\Gamma_n^k$ are positive for $0.5 < H < 1$; the sum of $\Gamma_n^k$ for any fixed $n$ tends to 1 as $H \uparrow 1$.

\setlength{\tabcolsep}{3.25pt}

\begin{table}[ht]
\caption{Coefficients $Q_n^k$ for $H=0.51$.}\label{tab:0.51}
\centering
\begin{tabular}{*{11}{l}}					
  \toprule\footnotesize
$\bm{n\backslash k}$ & \textbf{1} & \textbf{2} & \textbf{3} & \textbf{4} & \textbf{5} & \textbf{6} & \textbf{7} & \textbf{8} & \textbf{9} & \textbf{10} \\
  \midrule
1 & 0.013884 & & & & & & & & & \\
 2 & 0.013795 & 0.005149 & & & & & & & & \\
 3 & 0.013769 & 0.005096 & 0.003342 & & & & & & & \\
 4 & 0.013756 & 0.005079 & 0.003304 & 0.002480 & & & & & & \\
 5 & 0.013747 & 0.005070 & 0.003292 & 0.002451 & 0.001973 & & & & & \\
 6 & 0.013742 & 0.005064 & 0.003284 & 0.002441 & 0.001949 & 0.001638 & & & & \\
 7 & 0.013738 & 0.005059 & 0.003279 & 0.002435 & 0.001940 & 0.001617 & 0.001400 & & & \\
 8 & 0.013735 & 0.005056 & 0.003276 & 0.002431 & 0.001935 & 0.001610 & 0.001383 & 0.001223 & &  \\
 9 & 0.013732 & 0.005054 & 0.003273 & 0.002428 & 0.001932 & 0.001606 & 0.001376 & 0.001207 & 0.001085 & \\
 10 & 0.013730 & 0.005052 & 0.003271 & 0.002425 & 0.001929 & 0.001603 & 0.001372 & 0.001202 & 0.001071 & 0.000975 \\
\bottomrule
\end{tabular}
\end{table}

\begin{table}[ht]
\caption{Coefficients $Q_n^k$ for $H=0.6$.}
\centering
\begin{tabular}{*{11}{l}}					
  \toprule\small
$\bm{n\backslash k}$ & \textbf{1} & \textbf{2} & \textbf{3} & \textbf{4} & \textbf{5} & \textbf{6} & \textbf{7} & \textbf{8} & \textbf{9} & \textbf{10} \\
  \midrule
1 & 0.138815 & & & & & & & & &  \\
 2 & 0.130739 & 0.043422 & & & & & & & &  \\
 3 & 0.128648 & 0.038862 & 0.028344 & & & & & & & \\
 4 & 0.127592 & 0.037639 & 0.025192 & 0.020583 & & & & & & \\
 5 & 0.126965 & 0.036969 & 0.024314 & 0.018184 & 0.016087 & & & & & \\
 6 & 0.126552 & 0.036549 & 0.023814 & 0.017503 & 0.014158 & 0.013157 & & & & \\
 7 & 0.126260 & 0.036261 & 0.023490 & 0.017104 & 0.013603 & 0.011548 & 0.011103 & & & \\
 8 & 0.126044 & 0.036051 & 0.023262 & 0.016841 & 0.013272 & 0.011081 & 0.009725 & 0.009585 & & \\
 9 & 0.125878 & 0.035892 & 0.023092 & 0.016653 & 0.013051 & 0.010799 & 0.009322 & 0.008382 & 0.008419 & \\
 10 & 0.125746 & 0.035767 & 0.022962 & 0.016511 & 0.012891 & 0.010609 & 0.009077 & 0.008029 & 0.007354 & 0.007498 \\
\bottomrule
\end{tabular}
\end{table}

\begin{table}[ht]
\caption{Coefficients $Q_n^k$ for $H=0.7$.}
\centering
\begin{tabular}{*{11}{l}}					
  \toprule\small
$\bm{n\backslash k}$ & \textbf{1} & \textbf{2} & \textbf{3} & \textbf{4} & \textbf{5} & \textbf{6} & \textbf{7} & \textbf{8} & \textbf{9} & \textbf{10} \\
  \midrule
1 & 0.268776 & & & & & & & & &  \\
 2 & 0.242278 & 0.067642 & & & & & & & & \\
 3 & 0.236101 & 0.052812 & 0.045780 & & & & & & & \\
 4 & 0.233065 & 0.049661 & 0.035927 & 0.032171 & & & & & & \\
 5 & 0.231346 & 0.047932 & 0.033750 & 0.024852 & 0.024639 & & & & & \\
 6 & 0.230255 & 0.046897 & 0.032493 & 0.023212 & 0.018868 & 0.019801 & & & & \\
 7 & 0.229509 & 0.046210 & 0.031715 & 0.022234 & 0.017561 & 0.015065 & 0.016460 & & & \\
 8 & 0.228971 & 0.045725 & 0.031184 & 0.021615 & 0.016765 & 0.013985 & 0.012463 & 0.014025 & & \\
 9 & 0.228567 & 0.045365 & 0.030800 & 0.021185 & 0.016254 & 0.013317 & 0.011546 & 0.010579 & 0.012178 &  \\
 10 & 0.228255 & 0.045089 & 0.030510 & 0.020869 & 0.015894 & 0.012883 & 0.010973 & 0.009785 & 0.009158 & 0.010733 \\
\bottomrule
\end{tabular}
\end{table}

\begin{table}
\caption{Coefficients $Q_n^k$ for $H=0.8$.}
\centering
\begin{tabular}{*{11}{l}}					
  \toprule\small
$\bm{n\backslash k}$ & \textbf{1} & \textbf{2} & \textbf{3} & \textbf{4} & \textbf{5} & \textbf{6} & \textbf{7} & \textbf{8} & \textbf{9} & \textbf{10} \\
  \midrule
1 & 0.376892 & & & & & & & & & \\
 2 & 0.332755 & 0.073057 & & & & & & & &  \\
 3 & 0.323218 & 0.046763 & 0.053946 & & & & & & & \\
 4 & 0.318643 & 0.042641 & 0.037340 & 0.036154 & & & & & & \\
 5 & 0.316190 & 0.040274 & 0.034583 & 0.024079 & 0.027185 & & & & & \\
 6 & 0.314689 & 0.038939 & 0.032895 & 0.022066 & 0.017849 & 0.021453 & & & & \\
 7 & 0.313695 & 0.038083 & 0.031908 & 0.020780 & 0.016283 & 0.013915 & 0.017566 & & & \\
 8 & 0.312999 & 0.037495 & 0.031255 & 0.020011 & 0.015254 & 0.012646 & 0.011291 & 0.014771 & & \\
 9 & 0.312489 & 0.037071 & 0.030796 & 0.019492 & 0.014628 & 0.011794 & 0.010232 & 0.009428 & 0.012677 & \\
 10 & 0.312103 & 0.036753 & 0.030457 & 0.019121 & 0.014200 & 0.011271 & 0.009510 & 0.008524 & 0.008045 & 0.011056 \\
\bottomrule
\end{tabular}
\end{table}

\begin{table}
\caption{Coefficients $Q_n^k$ for $H=0.9$.}
\centering
\begin{tabular}{*{11}{l}}					
  \toprule\small
$\bm{n\backslash k}$ & \textbf{1} & \textbf{2} & \textbf{3} & \textbf{4} & \textbf{5} & \textbf{6} & \textbf{7} & \textbf{8} & \textbf{9} & \textbf{10} \\
  \midrule
 1 & 0.454626 & & & & & & & & &  \\
 2 & 0.403922 & 0.062598 & & & & & & & &  \\
 3 & 0.393175 & 0.026612 & 0.055279 & & & & & & & \\
 4 & 0.388117 & 0.022914 & 0.034159 & 0.034475 & & & & & & \\
 5 & 0.385560 & 0.020488 & 0.031750 & 0.019333 & 0.025609 & & & & & \\
 6 & 0.384053 & 0.019229 & 0.030042 & 0.017646 & 0.014144 & 0.019817 & & & & \\
 7 & 0.383088 & 0.018445 & 0.029120 & 0.016370 & 0.012869 & 0.010706 & 0.015985 & & & \\
 8 & 0.382430 & 0.017924 & 0.028526 & 0.015666 & 0.011862 & 0.009697 & 0.008507 & 0.013264 & & \\
 9 & 0.381961 & 0.017557 & 0.028120 & 0.015202 & 0.011298 & 0.008875 & 0.007682 & 0.006973 & 0.011252 & \\
 10 & 0.381613 & 0.017289 & 0.027828 & 0.014879 & 0.010921 & 0.008409 & 0.006992 & 0.006281 & 0.005857 & 0.009711 \\
\bottomrule
\end{tabular}
\end{table}
\clearpage
\begin{table}
\caption{Coefficients $Q_n^k$ for $H=0.99$.}\label{tab:0.99}
\centering
\begin{tabular}{*{11}{l}}					
  \toprule\small
$\bm{n\backslash k}$ & \textbf{1} & \textbf{2} & \textbf{3} & \textbf{4} & \textbf{5} & \textbf{6} & \textbf{7} & \textbf{8} & \textbf{9} & \textbf{10} \\
  \midrule
1 & 0.496847 & & & & & & & & & \\
 2 & 0.454561 & 0.043067 & & & & & & & & \\
 3 & 0.444715 & 0.000938 & 0.052726 & & & & & & & \\
 4 & 0.440112 & $-0.001495$ & 0.030379 & 0.029737 & & & & & & \\
 5 & 0.437912 & $-0.003621$ & 0.028833 & 0.013718 & 0.022123 & & & & & \\
 6 & 0.436656 & $-0.004607$ & 0.027350 & 0.012702 & 0.010257 & 0.016765 & & & & \\
 7 & 0.435877 & $-0.005212$ & 0.026630 & 0.011610 & 0.009520 & 0.007463 & 0.013351 & & & \\
 8 & 0.435359 & $-0.005601$ & 0.026173 & 0.011071 & 0.008669 & 0.006900 & 0.005813 & 0.010945 & & \\
 9 & 0.434998 & $-0.005868$ & 0.025869 & 0.010719 & 0.008243 & 0.006212 & 0.005366 & 0.004673 & 0.009187 & \\
 10 & 0.434735 & $-0.006059$ & 0.025656 & 0.010482 & 0.007961 & 0.005865 & 0.004794 & 0.004307 & 0.003862 & 0.007854 \\
\bottomrule
\end{tabular}
\end{table}

%%%%%%%%%%%%% Norm of projections

%\newpage
\section{Asymptotic behavior of the projection norm}\label{sec5}

In this section we investigate the asymptotic behavior of the norms of projections \eqref{projectgamma1} and \eqref{projectn} with the growth of $H$ and $n$, i.e.,
\begin{gather*}
 R_1(n) = \ex\left|\ex (\Delta_1\mid\Delta_2,\ldots, \Delta_n) \right|^2,\\
 R_2(n) = \ex\left|\ex(\Delta_{n}|\Delta_{1},\ldots,\Delta_{n-1},\Delta_{n+1},\ldots,\Delta_{2n-1}) \right|^2.
\end{gather*}

We can represent these norms as follows.
\begin{proposition}
The norms $R_1(n)$ and $R_2(n)$ admit the following representations:
\begin{equation}\label{eq:norms}
R_1(n) = \sum_{k=2}^n\Gamma_n^k\rho_{k-1},
\qquad
R_2(n) = 2\sum_{k=1}^{n-1} Q_{n-1}^k \rho_{k}.
\end{equation}
\end{proposition}

\begin{proof}
By the definition of the norm $R_1(n)$ and the representation \eqref{projectgamma1}, we have
\begin{gather*}
 R_1(n) = \ex\left|\sum_{k=2}^n \Gamma_n^k \Delta_k \right|^2
 =\sum_{k,l=2}^n \Gamma_n^k \Gamma_n^l \rho_{|k-l|}.
\end{gather*}
Now taking into account \eqref{eq:gamma-sys}, we obtain the first identity in \eqref{eq:norms}.

The identity for $R_2(n)$ can be proved similarly.
By \eqref{projectj},
\begin{equation}\label{eq:R2norm}
\begin{split}
 R_2(n) &=\ex\left|\sum_{k=1}^{n-1} Q_{n-1}^k
 (\Delta_{n-k}+\Delta_{n+k}) \right|^2
 =\sum_{k,l=1}^{n-1} Q_{n-1}^k Q_{n-1}^l
 \ex(\Delta_{n-k}+\Delta_{n+k})(\Delta_{n-l}+\Delta_{n+l})\\
 &%=2\sum_{k,l=1}^{n-1} Q_{n-1}^k Q_{n-1}^l
 %\left(\rho_{|k-l|}+\rho_{k+l}\right)
 =2\sum_{k=1}^{n-1} Q_{n-1}^k \sum_{l=1}^{n-1}  Q_{n-1}^l
 \left(\rho_{|k-l|}+\rho_{k+l}\right),
\end{split}
\end{equation}
Observe that by \eqref{systq}
\[
 \rho_{k}=
 \sum_{j=1}^{n-1}Q_{n-1}^{n-j}\left(\rho_{|j-n+k|}+\rho_{|n-j+k|}\right)
 =
 \sum_{l=1}^{n-1}Q_{n-1}^{l}\left(\rho_{|k-l|}+\rho_{k+l}\right),
 \quad 1\le k\le n-1.
\]
We now substitute this identity into the right-hand side of \eqref{eq:R2norm}
and arrive to the desired formula for $R_2(n)$.
\end{proof}

We numerically study the behavior of the norms $R_1(n)$ and $R_2(n)$. Figure \ref{fig:norms1} illustrates their behavior as functions of $n$. It is observed that both norms increase with increasing
$n$ or $H$. Additionally, as $n\to\infty$, $R_1(n)$ and $R_2(n)$ approach certain limits that depend on $H$ and can be calculated numerically.

%With all these previous results in mind, in this paper we considered three main tasks: to proceed analytically with the coefficient

\begin{figure}
\includegraphics[width=0.4\linewidth]{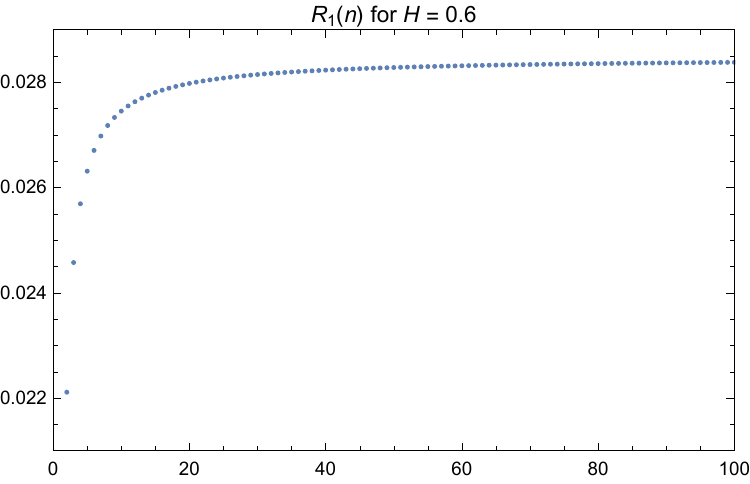}
\includegraphics[width=0.4\linewidth]{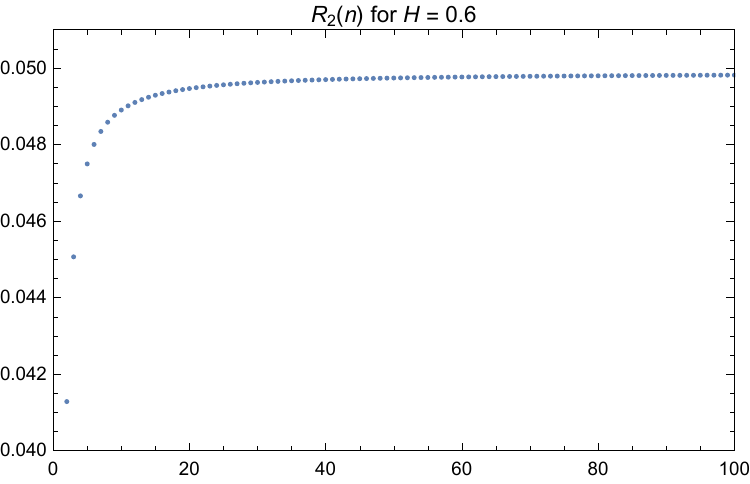}
\\
\includegraphics[width=0.4\linewidth]{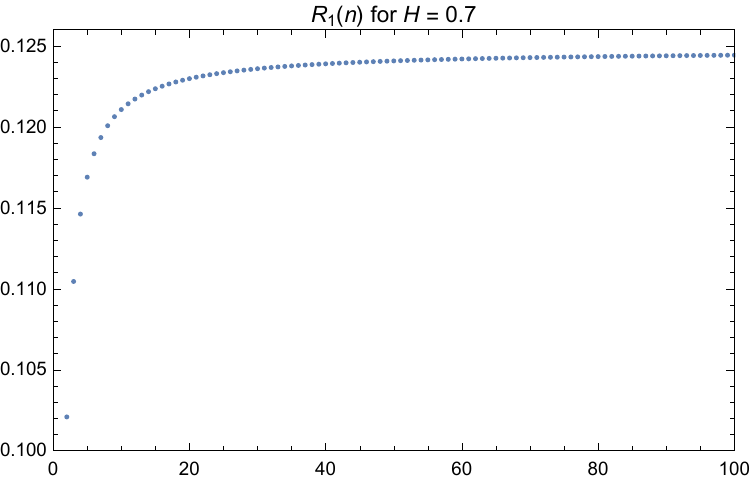}
\includegraphics[width=0.4\linewidth]{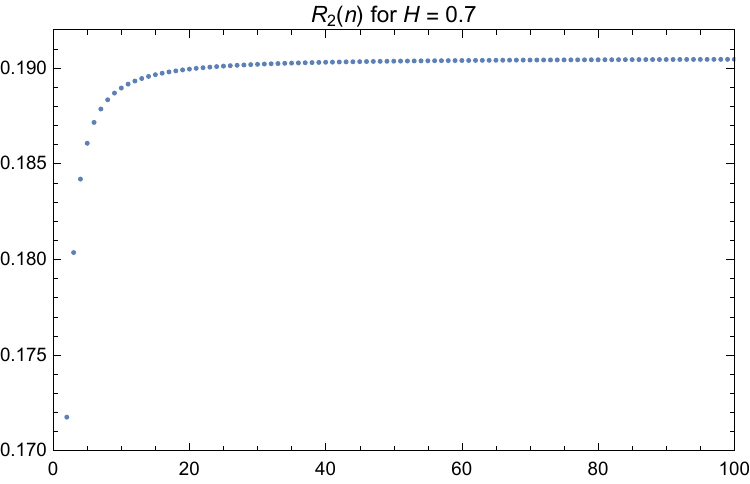}
\\
\includegraphics[width=0.4\linewidth]{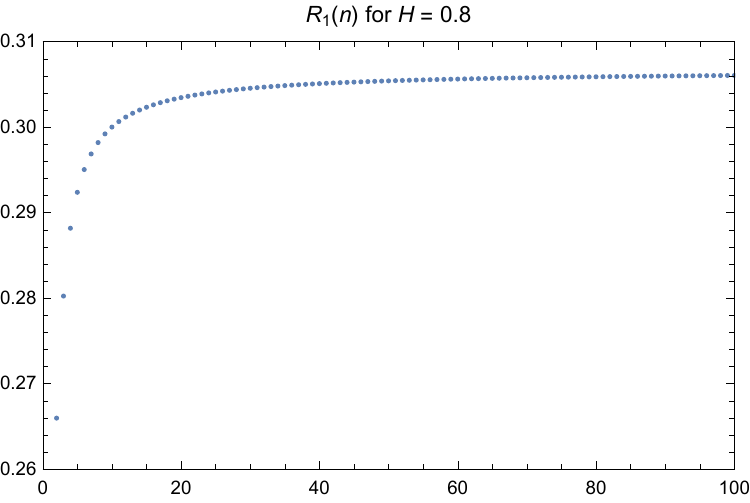}
\includegraphics[width=0.4\linewidth]{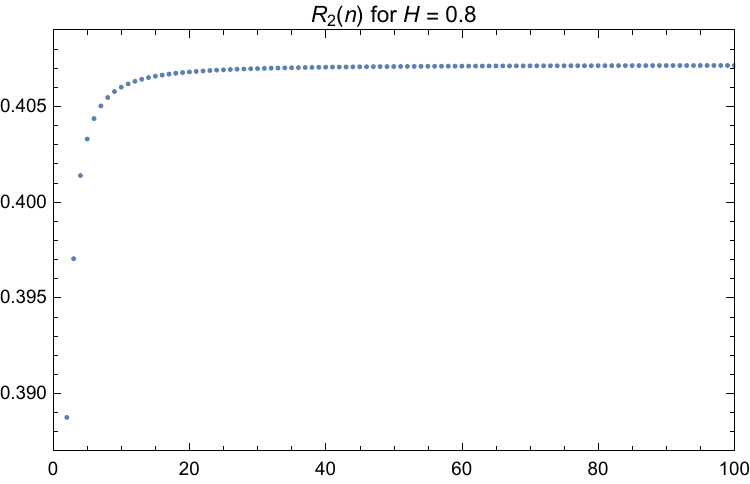}
\\
\includegraphics[width=0.4\linewidth]{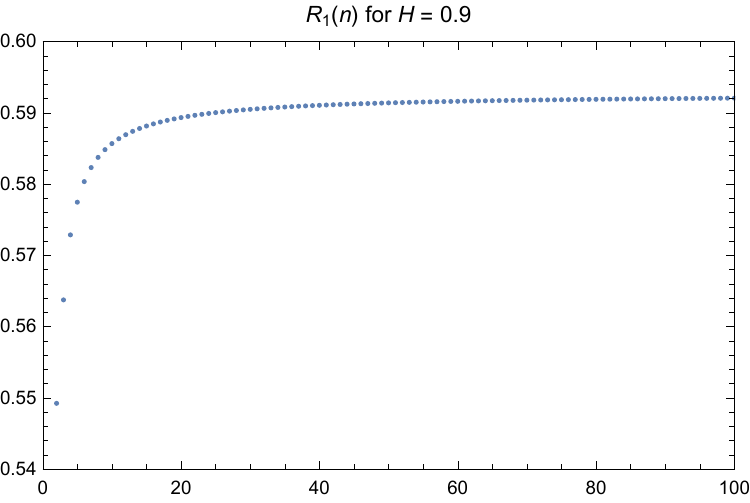}
\includegraphics[width=0.4\linewidth]{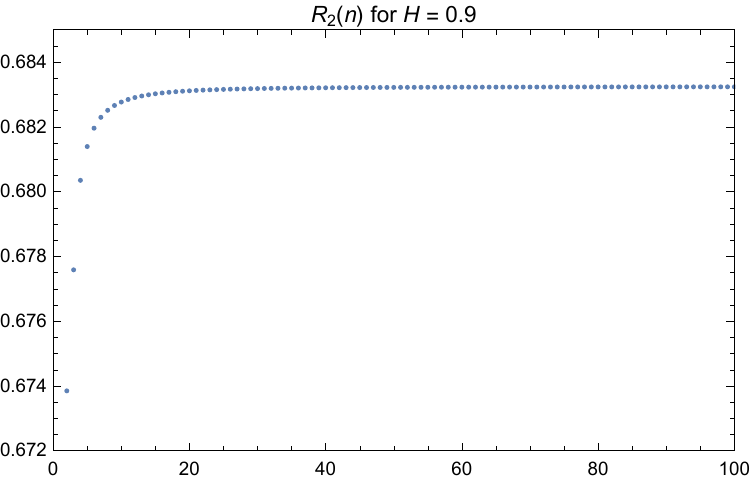}
\caption{Norms $R_1(n)$ (left) and $R_2(n)$ (right) as functions of $n$ for $H = 0.6, 0.7, 0.8, 0.9$}\label{fig:norms1}
\end{figure}

Table \ref{tab:norms} presents the values of the norms for various $n$ and $H$. For fixed $n$ and $H$, it is evident that
$R_2(n)$ is greater than $R_1(n)$.

\setlength{\tabcolsep}{4.5pt}
\begin{table}[h]
\caption{Norms $R_1(n)$ and $R_2(n)$ for various $n$}\label{tab:norms}
\small
    \begin{tabular}{cc*{7}{l}}	\toprule
      & $\bm n$ & \textbf{2} & \textbf{50} & \textbf{100} & \textbf{200} & \textbf{300} & \textbf{400} \\ \midrule
      \multirow{2}{*}{$H=0.6$} & $R_1(n)$ & 0.022111  & 0.028281 & 0.028380 & 0.028429 & 0.028445 & 0.028453 \\
      & $R_2(n)$ & 0.041283 & 0.049737 & 0.049813 & 0.049846 & 0.049856 & 0.049860  \\
      \midrule
      \multirow{2}{*}{$H=0.7$} & $R_1(n)$ & 0.102085 & 0.124071 & 0.124427 & 0.124604 & 0.124662 & 0.124692  \\
      & $R_2(n)$ & 0.171752 & 0.190355 & 0.190451 & 0.190487 & 0.190496 & 0.190500 \\
      \midrule
      \multirow{2}{*}{$H=0.8$} & $R_1(n)$ & 0.265964 &  0.305413 & 0.306050 &  0.306365 & 0.306469 & 0.306521  \\
      & $R_2(n)$ & 0.388739 & 0.407076 & 0.407131 & 0.407149 & 0.407154 & 0.407155 \\
      \midrule
      \multirow{2}{*}{$H=0.9$} & $R_1(n)$ & 0.549231 & 0.591406 & 0.592072 & 0.592402 & 0.592511 & 0.592565 \\
      & $R_2(n)$ & 0.673847 &  0.683218 & 0.683235 & 0.68324 & 0.683241 & 0.683242 \\
      \midrule
      \multirow{2}{*}{$H=0.99$} & $R_1(n)$ & 0.945689 & 0.953188 & 0.953302 & 0.953359 & 0.953378 & 0.953387 \\
      & $R_2(n)$ & 0.966334 &  0.967078 & 0.967079 & 0.967079 & 0.967079 & 0.967079 \\
      \bottomrule
    \end{tabular}
\end{table}

Figure \ref{fig:norms2} contains the graphs of the norms $R_1(n)$ and $R_2(n)$ as functions $H$ for $n=500$.
\clearpage
\begin{figure}
\includegraphics[scale=0.75]{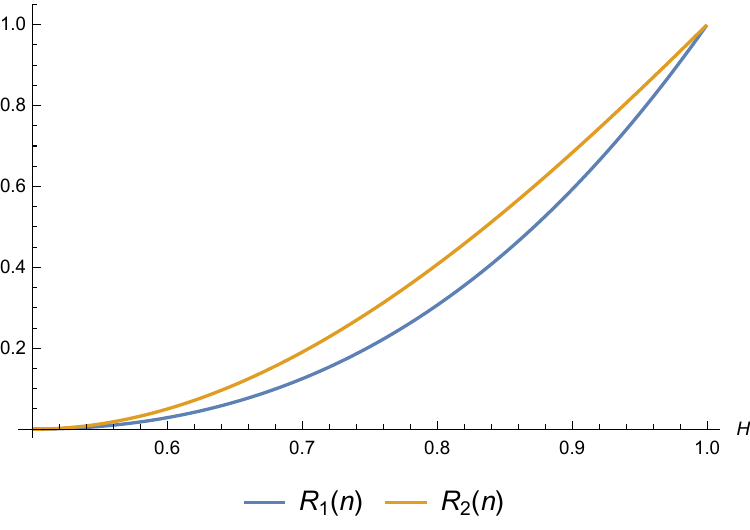}
\caption{Norms $R_1(n)$ and $R_2(n)$ for $n=500$}\label{fig:norms2}
\end{figure}

%%%%%%%%%%%%%%%%%%%%%%%%%%%%%%%%%%%%%%%%%%%%
\bibliographystyle{amsplain}
\bibliography{fgnbib}

\providecommand{\bysame}{\leavevmode\hbox to3em{\hrulefill}\thinspace}
\providecommand{\MR}{\relax\ifhmode\unskip\space\fi MR }
\providecommand{\MRhref}[2]{%
  \href{http://www.ams.org/mathscinet-getitem?mr=#1}{#2}
}
\providecommand{\href}[2]{#2}
\begin{thebibliography}{10}

\bibitem{bamish}
O.~Banna, Y.~Mishura, K.~Ralchenko, and S.~Shklyar, \emph{Fractional {B}rownian
  motion. {A}pproximations and projections}, Wiley-ISTE, 2019.

\bibitem{barton}
R.~J. Barton and H.~V. Poor, \emph{Signal detection in fractional {Gaussian}
  noise}, IEEE Transactions on Information Theory \textbf{34} (1988), no.~5,
  943--959.

\bibitem{beran}
J.~Beran, \emph{Statistics for long-memory processes}, 1 ed., Monographs on
  Statistics and Applied Probability, Chapman \& Hall, New York, 1994.

\bibitem{bodnmish}
I.~Bodnarchuk and Y.~Mishura, \emph{Combinatorial approach to the calculation
  of projection coefficients for the simplest {G}aussian--{V}olterra process},
  Modern Stochastics: Theory and Applications (2024), 1--17.

\bibitem{davalos}
A.~D´avalos, M.~Jabloun, P.~Ravier, and O.~Buttelli, \emph{Theoretical study
  of multiscale permutation entropy on finite-length fractional {Gaussian}
  noise}, 2018 26th European Signal Processing Conference (EUSIPCO), IEEE,
  2018, pp.~1087--1091.

\bibitem{deli}
D.~Delignieres, \emph{Correlation properties of (discrete) fractional
  {Gaussian} noise and fractional {Brownian} motion}, Mathematical Problems in
  Engineering \textbf{2015} (2015), no.~1, 485623.

\bibitem{li1}
M.~Li\bs{1}, X.~Sun, and X.~Xiao, \emph{Revisiting fractional {Gaussian}
  noise}, Physica A: Statistical Mechanics and its Applications \textbf{514}
  (2019), 56--62.

\bibitem{li2}
M.~Li\bs{2}, \emph{Modified multifractional {Gaussian} noise and its
  application}, Physica Scripta \textbf{96} (2021), no.~12, 125002.

\bibitem{LiptserShiryaev2013}
R.~S. Liptser and A.~N. Shiryaev, \emph{Statistics of random processes {II}:
  {A}pplications}, 2 ed., Springer, 2013.

\bibitem{liu}
Y.~Liu, Y.~Liu, K.~Wang, T.~Jiang, and L.~Yang, \emph{Modified periodogram
  method for estimating the {H}urst exponent of fractional {Gaussian} noise},
  Physical Review E \textbf{80} (2009), no.~6, 066207.

\bibitem{malyarenko}
A.~Malyarenko, Y.~Mishura, K.~Ralchenko, and S.~Shklyar, \emph{Entropy and
  alternative entropy functionals of fractional {Gaussian} noise as the
  functions of {Hurst} index}, Fractional Calculus and Applied Analysis
  \textbf{26} (2023), no.~3, 1052--1081.

\bibitem{mandelbrot1}
B.~B. Mandelbrot\bs{1} and J.~W.~Van Ness, \emph{Fractional {Brownian} motions,
  fractional noises and applications}, SIAM review \textbf{10} (1968), no.~4,
  422--437.

\bibitem{mandelbrot2}
B.~B. Mandelbrot\bs{2}, \emph{A fast fractional {Gaussian} noise generator},
  Water Resources Research \textbf{7} (1971), no.~3, 543--553.

\bibitem{meerson}
B.~Meerson, O.~B\'enichou, and G.~Oshanin, \emph{Path integrals for fractional
  {B}rownian motion and fractional {G}aussian noise}, Physical Review E
  \textbf{106} (2022), L062102.

\bibitem{MishuraRalchenkoSchilling2022}
Y.~Mishura, K.~Ralchenko, and R.~L. Schilling, \emph{Analytical and
  computational problems related to fractional {Gaussian} noise}, Fractal and
  Fractional \textbf{6} (2022), no.~11, 1--22.

\bibitem{risks}
Y.~Mishura, K.~Ralchenko, and S.~Shklyar, \emph{General conditions of weak
  convergence of discrete-time multiplicative scheme to asset price with
  memory}, Risks \textbf{8} (2020), no.~1, 11.

\bibitem{MishuraZili2018}
Y.~Mishura and M.~Zili, \emph{Stochastic analysis of mixed fractional
  {Gaussian} processes}, ISTE Press -- Elsevier, 2018.

\bibitem{molz}
F.~J. Molz, H.~H. Liu, and J.~Szulga, \emph{Fractional {Brownian} motion and
  fractional {Gaussian} noise in subsurface hydrology: A review, presentation
  of fundamental properties, and extensions}, Water Resources Research
  \textbf{33} (1997), no.~10, 2273--2286.

\bibitem{qian}
H.~Qian, \emph{Fractional {B}rownian motion and fractional {G}aussian noise},
  Processes with Long-Range Correlations: Theory and Applications, Springer,
  2003, pp.~22--33.

\bibitem{stratonovich}
R.~L. Stratonovich, \emph{Theory of information and its value}, Springer Nature
  Switzerland AG, Cham, Switzerland, 2020, Edited by R. V. Belavkin, P. M.
  Pardalos and J. C. Principe.

\bibitem{wang}
W.~Wang, A.~G. Cherstvy, X.~Liu, and R.~Metzler, \emph{Anomalous diffusion and
  nonergodicity for heterogeneous diffusion processes with fractional
  {Gaussian} noise}, Physical Review E \textbf{102} (2020), no.~1, 012146.

\bibitem{zunino}
L.~Zunino, D.~G. Pérez, M.~T. Martín, M.~Garavaglia, A.~Plastino, and O.~A.
  Rosso, \emph{Permutation entropy of fractional {Brownian} motion and
  fractional {Gaussian} noise}, Physics Letters A \textbf{372} (2008),
  no.~27-28, 4768--4774.

\end{thebibliography}

\end{document}